\documentclass[a4paper]{amsart}
 \usepackage{graphicx}
\newcommand{\R}{\mathbb{R}} 

\newcommand{\Dcal}{\mathcal{D}}

\newcommand{\Ecal}{\mathcal{E}}

\newcommand{\Hcal}{\mathcal{H}}
\newcommand{\ud}{\frac{1}{2}}
\newcommand{\sgn}{{\rm sgn}}

\newcommand{\rot}{{\rm curl}}
\newcommand{\di}{{\rm div}}
\newcommand{\pv}{\times}

\newcommand{\choc}{S}
\newcommand{\rare}{R}
\newcommand{\ki}{\chi}
\newcommand{\eps}{\tau}
\newcommand{\E}{{\bf P}}

\newcommand{\nn}{|E|^2}

\newcommand{\D}{{\bf D}}
\newcommand{\om}{\omega}

\newcommand{\ds}{\displaystyle}
\def \dx{\Delta x}
\def \dy{\Delta y}
\def \dt{\Delta t}
\def \ud{\frac{1}{2}}

\def \una{u^n_\alpha}
\def \unb{u^n_\beta}

\def \unpa{u^{n+1}_\alpha}

\def \f2{\vec{{F}}}
\def \nor{\vec{n}}

\newtheorem{theorem}{Theorem}[section]
\newtheorem{lemma}[theorem]{Lemma}
\newtheorem{prop}[theorem]{Proposition}

\theoremstyle{definition}
\newtheorem{definition}[theorem]{Definition}

\theoremstyle{remark}

\numberwithin{equation}{section}

\begin{document}
\title[\sc Godunov scheme for Kerr system.]{Godunov scheme for Maxwell's equations with Kerr nonlinearity.}
\author{Denise Aregba-Driollet}
\address{Institut de Math\'ematiques de Bordeaux, UMR 5251, 
351 cours de la lib\'eration, 33405 Talence Cedex, France. {\it aregba@math.u-bordeaux1.fr}} 
\date{\today}


\begin{abstract} 
We study the Godunov scheme for a nonlinear Maxwell model arising in nonlinear optics, the Kerr model. This is a hyperbolic system of conservation laws with some eigenvalues of variable multiplicity, neither genuinely nonlinear nor linearly degenerate. The solution of the Riemann problem for the full-vector $6 \times 6$ system is constructed and proved to exist for all data. This solution is compared to the one of the reduced Transverse Magnetic model. The scheme is implemented in one and two space dimensions. The results are very close to the ones obtained with a Kerr-Debye relaxation approximation.
\end{abstract}
\maketitle
\section{Introduction}
In nonlinear optics, the propagation of electromagnetic waves in a crystal can be modelized by the so-called Kerr and Kerr-Debye models. Denoting $E$ and $H$ the electric and magnetic fields, $D$ and $B$ the electric and
magnetic displacements, one writes the tridimensional Maxwell's equations
\[
\left\{
\begin{array}{lcl}
\partial_t D - \rot H = 0,\\
\partial_t B + \rot E =0,
\end{array}
\right.
\]
with $\di D=\di B=0$, and the constitutive relations
\[
\left\{
\begin{array}{lcl}
B&=&\mu_0 H \\
D&=&\epsilon_0 E +P
\end{array}
\right.
\]
where $P$ is the nonlinear polarization and $\mu_0$, $\epsilon_0$ are
the free space permeability and permittivity. 

If the medium exhibits an instantaneous response, then one can use a Kerr model \begin{equation}\label{PK}
P=P_K=\epsilon_0 \epsilon_r |E|^2 E,
\end{equation}
where $\epsilon_r$ is the relative permittivity. See for example \cite{shen} for further details. In that case, Maxwell's equations read as a $6 \times 6$ quasilinear system of conservation laws: 
\begin{equation}\label{Kerr3D}
\left\{
\begin{array}{lcl}
\partial_t D - \rot H = 0,\\
\partial_t H + \mu_0^{-1}\rot (\E(D)) =0
\end{array}
\right.
\end{equation}
where $\E$ is the reciprocal function of $\D$: 
\[ \D(E)=\epsilon_0(1+\epsilon_r |E|^2)E.\]
Denoting 
\begin{equation}\label{defq}
 q(e)=\epsilon_0(e+\epsilon_r e^3), \qquad e\in \R, \qquad p=q^{-1},
\end{equation}
we have 
\begin{equation}\label{defP}
E= \E (D)=\frac{D}{\epsilon_0 (1+\epsilon_r p^2(|D|))}\, , \quad |E|=p(|D|).
\end{equation}
If $(D,H)$ is solution of (\ref{Kerr3D}) then $\partial_t (\di D)=\partial_t (\di (\mu_0 H))=0$ so at the theoritical level the divergence conditions have to be satisfied for the initial data only.

To solve numerically Maxwell models involved in nonlinear optics, it is rather classical to use a Finite Difference Time-Domain (FDTD) method introduced by K.S. Yee \cite{yee}. In a related context, we refer the reader to the works by R. W. Ziolkowski {\it et al} \cite{ziol-jud}, \cite{ziol}, A. Bourgeade {\it et al} \cite{bour-fr}, \cite{bour-saut}, O. Saut \cite{saut}. Finite element methods can also be adapted, see \cite{thhuyn}. Finite volumes are used by A. de la Bourdonnaye with a third order Roe solver for a Kerr model \cite{de-la-B}, and by M. Kanso for a linearly degenerate Kerr-Debye model (see below) \cite{kanso}.

Our aim here is to construct an accurate and efficient scheme for Kerr model (\ref{Kerr3D}). In particular, we have to be able to approximate the shocks which, even with smooth initial data, can appear in finite time, see \cite{c-h2}. For this purpose, in the framework of finite volumes, we are going to construct the Godunov scheme for system (\ref{Kerr3D}) in one and two space dimensions. 

As well known, the solution of the Riemann problem is the cornerstone of Godunov scheme. Consider a system of conservation laws 
\[ \partial_t u_j + \di \f2_j(u)=0, \quad 1 \leq j \leq N.\] 
Let $\{C_\alpha, \; \alpha \in A\}$ be an admissible mesh of the computational domain, and let us denote $\Gamma_{\alpha \beta}$ the common edge of $C_\alpha$ and $C_\beta$, and $\nor_{\alpha \beta}$ the unitary normal vector to $\Gamma_{\alpha \beta}$, pointing from $C_\alpha$ to $C_\beta$. The approximation $\unpa$ of $u(.,t_{n+1})$ on $C_\alpha$ is computed as follows:
\begin{equation}\label{int-VF-ns}
\unpa=\una - \frac{\dt}{|C_\alpha|}\sum_{\beta,C_{\alpha}\cap C_{\beta} \not=\emptyset} \Phi(\una,\unb,\nor_{\alpha \beta})\,|\Gamma_{\alpha \beta} |,
\end{equation}
the numerical flux function $\Phi$ being defined by
\begin{equation}\label{phigen} 
\Phi_j(u,v,\nor)=\f2_j({\overline w}(0))\cdot \nor, \quad  1 \leq j \leq N,
\end{equation}
and ${\overline w}(\frac{y}{t})=w(y,t)$ is the solution of the one-dimensional Riemann problem
\[ 
\left\{
\begin{array}{l} 
\partial_t w_j + \partial_y (\f2_j(w)\cdot \nor)=0,\quad 1 \leq j \leq N,\\
w(y,0)=\left| \begin{array}{lr}
u&{\rm if} \; \; y<0, \\
v&{\rm if} \; \; y>0. 
\end{array}
\right.
\end{array}
\right.
\] 
Therefore, we have to solve the Riemann problem for Kerr system (\ref{Kerr3D}). As detailed hereafter, we have 4 linearly degenerate fields and 2 others are neither genuinely nonlinear, nor linearly degenerate, and the related eigenvalues own variable multiplicity. Hence the classical Lax existence results do not apply. In \cite{de-la-B},  a first existence result has been established for a reduced $4 \times 4$ case with the assumption that $D\cdot \nor=0$ and $H \cdot \nor =0$. In particular, the two-dimensional TM case, which is very important for the applications, does not enter this framework. Here, we deal with the full vector system and we implement the exact solution of the Riemann problem. 

In this article, we pay a particular attention to the 2D Transverse Magnetic (TM) case: for solutions depending on $x=(x_1,x_2)$, if one assumes that the data are such that $D_3=0$ and $H_1=H_2=0$, then so is the solution. Denoting $D=(D_1,D_2,0)$, (\ref{Kerr3D}) reduces to a $3 \times 3$ system: 
\begin{equation}\label{33TM}
\left\{
\begin{array}{lcl}
\partial_t D_1 - \partial_2H_3 = 0,\\
\partial_t D_2 + \partial_1H_3 = 0,\\
\partial_t H_3 + \mu_0^{-1}\left(\partial_1 (\E_2(D))- \partial_2 (\E_1(D))\right)=0.
\end{array}
\right.
\end{equation}
An even more particular case is the 1D setting with $D_1=0$ and $x=x_1$:
\begin{equation}\label{kerr}
\left\{
\begin{array}{lcl}
\partial_t D_2 + \partial_x H_3 &=&0, \\
\partial_t H_3 + \mu_0^{-1} \partial_x p(D_2) &=&0.
\end{array}
\right.
\end{equation}
It turns out that the 1D Kerr system (\ref{kerr}) is a so-called
p-system. As $p^\prime >0$ it is strictly hyperbolic but the
properties of the function $p$ differ from the ones which appear in
the general framework of gas dynamics or viscoelasticity \cite{Tz}. Here:
\[ p(0)=0, \quad p^\prime >0, \]
and $p$ is strictly convex on $]-\infty,0]$, strictly concave on $[0,+\infty[$.

The plan of the paper is the following. In section \ref{66-3D} we solve the Riemann problem for the $6 \times 6$ system (\ref{Kerr3D}): Lax solution is constructed and its existence and uniqueness are proved. Moreover, if the data are TM, so is the solution. 

In section \ref{2DTMsec}, we focus on the 2D TM system (\ref{33TM}). Here, we have to use Liu's condition (E) (\cite{Liu74}-\cite{Liu76}) for the admissibility of shocks and the solution of the Riemann problem. This solution is compared to the one obtained in section \ref{66-3D}. The mathematical entropy being the physical electromagnetic energy, it is proved that two distinct entropy solutions of (\ref{33TM}) (and (\ref{Kerr3D})) can exist. This may be surprising but we recall that no general uniqueness result is available for weak entropy solutions of systems of conservation laws. In the particular case of the Riemann problem, uniqueness theorems are proved only within a prescribed class of solutions, see \cite{serre}, \cite{Liu74}, and theorems \ref{exis-sol} and \ref{exis-sol2DTM}  here below.

Section \ref{num-sect} is devoted to numerical experiments. The $6 \times 6$ Riemann solver is implemented in one space dimension and then in a two-dimensional cartesian setting. Comparisons with exact solutions are performed. In case of non-uniqueness,  the computed solution is the Liu's one. Finally, a physically realistic case inspired from  \cite{ziol-jud} is analyzed. 

In each case, numerical comparison is done with a relaxation scheme obtained as follows: if the  medium exhibits a finite response time $\tau>0$, one should use the Kerr-Debye model for which
\begin{equation}\label{polar-KD}
P=P_{KD}=\epsilon_0  \chi E, \; \; \; \partial_t \chi + \frac{1}{\tau}
\chi= \frac{1}{\tau}\epsilon_r |E|^2 .
\end{equation}
Then one deals with a quasilinear hyperbolic system with source:
\begin{equation}\label{KD3D}
\left\{
\begin{array}{lr}
\partial_t D_\eps - \rot H_\eps = 0,& \\
\partial_t H_\eps + \mu_0^{-1}\rot E_\eps =0,& D_\eps=\epsilon_0(1+\chi_\eps) E_\eps \\
\partial_t \chi_\eps = \ds \frac{1}{\eps}\left(
\epsilon_r  |E_\eps|^2-\chi_\eps\right).&
\end{array}
\right.
\end{equation}
Let $U_\eps=(D_\eps,H_\eps,\chi_\eps )$ be a solution of (\ref{KD3D}). Formally, if $U_\eps \rightarrow U=(D,H,\chi)$ when $\tau$ tends to zero, then $U \in \mathcal{V}$ where $\mathcal{V}$ is the equilibrium manifold for the Kerr-Debye model:
\[ \mathcal{V}=\{ (D,H,\ki); \; \displaystyle \frac{\epsilon_r |D|^2 }{\epsilon_0^2
  (1+\ki)^{2}} - \ki=0 \}= \{ (D,H,\ki); \; \ki= \epsilon_r p^2(|D|)\}. \]
Therefore, $u=(D,H)$ is a solution of the Kerr system (\ref{Kerr3D}).

The Kerr-Debye model is a relaxation approximation of the Kerr
model and $\tau$ is the relaxation parameter. The Kerr system is the reduced system for the Kerr-Debye one in the sense of \cite{cll}, see also \cite{revnat} for a survey on hyperbolic relaxation problems.  In \cite{hh}, \cite{c-h2} some rigorous existence and convergence results are proved for Kerr-Debye system. In particular, for $\tau \not=0$, at least in certain configurations with smooth data, no shock is created.

Numerically, we take advantage of the fact that all the characteristic fields of (\ref{KD3D}) are linearly degenerate to design a scheme which owns a relaxed limit when $\tau=0$ and this limit is a consistent entropic approximation of (\ref{Kerr3D}). This method has been developed in \cite{kanso} for $2 \times 2$ and $3 \times 3$ cases. It is easy to compute the general case with the same ideas, see Annex. This gives us an explicit scheme, based on a physical model. In all cases,  the results are nearly the same as those of Godunov scheme, so that both method are proved to be efficient. This point is discussed in the conclusion.
\section{The Riemann problem for the full vector Kerr system} \label{66-3D}
In this part we solve the Riemann problem for system
(\ref{Kerr3D}). We denote $u=(D,H)$, $E=\E(D)$. For given $\omega \in \R^3$,
$|\omega|=1$, and $u_-, u_+ \in \R^6$, we fix the
initial data
\begin{equation}\label{riem-data}
u(x,0)=\left\{
\begin{array}{lr} 
u_-& if \quad x \cdot \omega <0, \\
u_+& if \quad x \cdot \omega >.
\end{array}
\right.
\end{equation}
We look for a selfsimilar entropy solution $u(x,t)=V(\frac{x \cdot \omega }{t})$ of (\ref{Kerr3D})(\ref{riem-data}). Denoting $y=x \cdot \omega$, we
therefore have to solve the Riemann problem for the one-dimensional
$6 \times 6$ system
\begin{equation}\label{Kerr-omeg}
\left\{
\begin{array}{lcl}
\partial_t D - \partial_y(\omega \pv H) = 0,\\
\partial_t H + \mu_0^{-1}\partial_y(\omega \pv \E(D)) =0.
\end{array}
\right.
\end{equation}
The admissible shocks of the Kerr
system have already been studied in \cite{dabh-CMS}. For the sake of completenes those
results are briefly recalled here. Then we construct the rarefaction
waves and we solve the whole Riemann problem.
\subsection{Characteristic fields of Kerr system, admissible shocks}\label{full-kerr-section}
Using the results of
\cite{dabh-CMS} we can state:
\begin{prop}\label{vpvp} \cite{dabh-CMS} 
The Kerr system (\ref{Kerr3D}) is hyperbolic diagonalizable: for
all $\omega \in \R^3$, $|\omega|=1$, the
eigenvalues of system (\ref{Kerr-omeg}) are given by 
\begin{equation}\label{6vp}
\lambda_{1}\leq \lambda_{2}=-\lambda< \lambda_{3}=\lambda_{4}=0 < \lambda_{5}=\lambda \leq  \lambda_{6}=-\lambda_{1}
\end{equation}
where $c=\sqrt{\epsilon_0 \mu_0}^{-1}$ is the light velocity, 
\begin{equation}\label{6vpvp}
\lambda_{1}^2=\frac{c^2}{1+\epsilon_r \nn}, \; \; \;
\lambda^2=c^2 \, \frac{1+\epsilon_r(\nn +2 (E\cdot
  \omega)^2)}{(1+\epsilon_r \nn)(1+3 \epsilon_r \nn)}.
\end{equation}
The inequalities in (\ref{6vp}) are strict if and only if $\omega \pv D \not= 0$.
\end{prop}
\begin{prop}\label{vnl} \cite{dabh-CMS} 
The characteristic fields 1,3,4,6 are linearly degenerate. \\ 
If $\om \pv D \not= 0$ the eigenvectors for $\lambda_2$ and $\lambda_5$ are: 
\[
 r_{i}(u,\omega)=\left(\begin{array}{c}
{\rm sgn}(\lambda_i) \om \pv (\om \pv D)  \\
-\lambda \, \omega \pv D 
\end{array}\right), \; \; \; i=2,5.
\]
The characteristic fields 2 and 5 are
genuinely nonlinear in the direction $\omega$ in the open set
\[ \Omega(\omega)=\{ (D,H)\in \R^6 \; ; \; \om \pv D \not=0  \}
\]
and for all $u\in \Omega(\omega)$ and $i\in \{2,5\}$
\begin{equation}\label{vnldef}
\lambda^\prime_{i}(u,\omega) \, r_{i}(u,\omega)>0.
\end{equation}
\end{prop}
We point out the fact that the fields 2 and 5 are neither genuinely nonlinear, nor linearly degenerate, so that the general theory about the resolution of the Riemann problem does not apply here. The characterization of admissible plane discontinuities is now
briefly recalled.

The Rankine-Hugoniot conditions for a discontinuity $(u_-,u_+)$
propagating with velocity $\sigma$ write
\begin{equation}\label{RH}
\begin{array}{lr}
\sigma [D]= - \omega \pv [H],& 
\sigma \mu_0 [H]= \omega \pv [E]
\end{array}
\end{equation}
where for a given quantity $v$, $[v]=v_+ - v_-$. \\
\noindent The divergence free conditions write
\begin{equation}\label{div0chocD}
\omega \cdot [D]=0,
\end{equation}
\begin{equation}\label{div0chocH}
\omega \cdot [H]=0.
\end{equation}
If $\sigma \not=0$, they are fulfilled as soon as (\ref{RH})
is satisfied.

\begin{prop}\label{ddc0} 
{\bf Stationary contact discontinuities.}
Stationary contact discontinuities are characterized by
\begin{equation}\label{rhstat}
\begin{array}{ll}
\omega \pv [H]=0,&
\omega \pv [E]=0.
\end{array}
\end{equation}
The divergence free ones are constant. 
\end{prop}
\begin{proof}
If (\ref{div0chocD})-(\ref{div0chocH}) are satisfied for a stationary shock, then
$[H]=0$. Let us now prove that $[D]=0$. According to Rankine-Hugoniot
conditions (\ref{rhstat}), we have 
\[ [\omega \pv(\omega \pv E)]=0.\]
Denoting $\gamma=\omega \pv(\omega \pv E_\pm)$, $e_\pm=E_\pm
\cdot \omega$, we have $D_\pm=\epsilon_0(1+\epsilon_r |E_\pm|^2)E_\pm$ and
\[
|E_\pm|^2=e_\pm^2+ \gamma^2
\]
so that $[D]\cdot \omega=0$ if and only if 
\[
e_+(1+\epsilon_r(\gamma^2+e_+^2))=e_-(1+\epsilon_r(\gamma^2+e_-^2))\]
which is equivalent to $e_+=e_-$. As a consequence, for a divergence
free stationary shock, one has $[E]=[(E\cdot \omega)\omega - \omega \pv(\omega
\pv E)]=0$ and thus $[D]=0$. 
\end{proof}
The fields 1 and 6 are linearly degenerate. The associated
contact discontinuities are characterized as follows:
\begin{prop}\label{ddc16}\cite{dabh-CMS} 
A discontinuity $\sigma$, $u_+$, $u_-$ is a contact discontinuity
associated to $\lambda_1$ or $\lambda_6$ if and only if 
\begin{equation}\label{defddc}
\left\{
\begin{array}{l}
   |E_+| = |E_-|  ,    \\
\sigma^2=c^2(1+\epsilon_r|E_+|^2)^{-1}=c^2(1+\epsilon_r|E_-|^2)^{-1},
\end{array}
\right.
\end{equation}
condition (\ref{div0chocD}) is satisfied, and 
\begin{equation}\label{HHHH}
[H]=\sigma \, \omega \pv [D].
\end{equation}
Moreover the only discontinuities satisfying Rankine-Hugoniot
conditions (\ref{RH}) and such that $|E_-|=|E_+|$ are the above
contact discontinuities. 
\end{prop}
At this point, it remains to study the discontinuities which are not
contact discontinuities. From now on we call shocks those discontinuities.

For a fixed left state $u_-$ the Hugoniot set of
$u_-$, denoted $\Hcal(u_-)$, is the set of the right states
$u_+$ such that there exists a shock connecting $u_-$ and $u_+$. We
denote then $\sigma=\sigma(u_+,u_-)$ the shock velocity. One can
 give a similar definition by fixing the right state. Moreover, we impose Lax admissibility conditions, which on the one hand ensure entropy dissipation, and on the other hand ensure that one can construct the solution of the Riemann problem as a superposition of simple waves. 
\begin{definition}\label{def-adm-lax}
A discontinuity $\sigma$, $u_-$, $u_+$ is a Lax k-shock if
\begin{equation}\label{lax-2}
\left\{
\begin{array}{l}
\lambda_k(u_+) \leq \sigma  \leq \lambda_{k+1}(u_+) \\
\lambda_{k-1}(u_-) \leq \sigma  \leq \lambda_{k}(u_-).
\end{array}
\right.
\end{equation}
\end{definition}
The following property holds:
\begin{prop}\label{symmetry-choc}
The Lax-admissible shocks are 2-shocks or 5-shocks. 

$\sigma$, $u_-=(D_-,H_-)$, $u_+=(D_+,H_+)$ is a Lax 2-shock if and
only if $-\sigma$, ${\overline u_-}=(-D_+,H_+)$, ${\overline u_+}=(-D_-,H_-)$ is a Lax 5-shock.
\end{prop}
Therefore we just give the results for Lax 2-shocks with a fixed left state $u_-$. In that goal, we define two functions:
\begin{equation}\label{fonctionf}
f(d,d_0)=\frac{c^2 \, d}{1+\epsilon_r
    p^2\left(\sqrt{d_0^2+d^2}\right)},\qquad d\, , d_0\in \R .
\end{equation}
When $d_0$ is fixed, $f(\cdot,d_0)$ is an increasing function, see \cite{dabh-CMS}. Hence we can define
\begin{equation}\label{saut-H-2choc-phi}
\choc(d_+,d_-,d_0) = \left( \left( f(d_+,d_0)-f(d_-,d_0) \right)(d_+-d_-) \right)^{\frac{1}{2}}, \qquad d_+ \, , d_-\, , d_0 \in \R.
\end{equation}
Two cases are under consideration.
\begin{prop}\label{chocs-vrais}\cite{dabh-CMS} 
{\bf Case $D_- \pv \omega\not = 0$.} 

Let $u_-=(D_-,H_-)$ be a fixed left state such that $D_- \pv
\omega\not = 0$. We denote 
\begin{equation}\label{zeta-5}
\zeta = - \, \frac {\omega \pv (\omega \pv D_-)}{|\omega \pv (\omega
  \pv D_-)|} \, .
\end{equation}
Then 
\[ D_-=d_0 \omega + d_- \, \zeta, \quad d_-= |\omega \pv (\omega
  \pv D_-)|>0.\]
The set $\Hcal_2(u_-)$ of the right states $u_+$ connected to $u_-$ by
a Lax 2-shock is a curve parametrized by $d_+\in \R$. It is the set of $(D_+,H_+) \in \R^6$ such that 
\[ 
D_+=d_0\, \omega + \, d_+ \, \zeta,  \quad 
H_+ - H_- = \choc(d_+,d_-,d_0) \om \pv  \zeta, \quad 0 \leq d_+ \leq d_-. 
\]
The shock velocity $\sigma$ satisfies $\sigma <0$ and
\begin{equation}\label{sigmaTM}
\sigma^2=\frac{f(d_+,d_0)-f(d_{-},d_0)}{d_+-d_{-}}\; .
\end{equation}
\end{prop}
\begin{prop}\label{chocs-ld}\cite{dabh-CMS} 
{\bf Case $D_- \pv \omega = 0$.} 

Let $u_-=(D_-,H_-)$ be a fixed left state such that $D_- \pv
\omega = 0$. 
Then the set $\Hcal(u_-)$ of the right states connected to $u_-$ by a
shock is the set of  $u_+=(D_+,H_+)$ satisfying (\ref{div0chocD}), (\ref{HHHH}) and
\begin{equation}\label{cld}
\sigma^2=\lambda_1^2(u_+)=c^2(1+\epsilon_r |E_+|^2)^{-1}.
\end{equation}
There is no nontrivial Lax 2-shock connecting $u_-$ to a right state
$u_+ \in \Hcal(u_-)$.
\end{prop}
\subsection{Rarefaction waves}\label{rarefaction-full-kerr}
We first determine the 2-rarefactions. The rarefaction waves are computed by using the integral curves of the
eigenvectors. As the 2-characteristic field is genuinely nonlinear in $\Omega(\om)=\{u =(D,H), \; \om
\pv D \not= 0\}$, the integral curves of $r_2$ allow us to determine a
rarefaction only in this open set. Those curves are the solutions of
the following differential system:
\begin{equation}\label{rare-full-sdo}
\left\{
\begin{array}{l}
D^\prime(\xi)=-\om \pv (\om \pv D(\xi))  \\
H^\prime(\xi)=-\lambda(D(\xi)) \om \pv D(\xi) .
\end{array}
\right.
\end{equation}
If $U=(D,H)$ is a solution of this system, then $D(\xi)\cdot
\omega$ and $H(\xi)\cdot \om$ are constant:
\[ D(\xi)\cdot \omega =d_0, \qquad H(\xi)\cdot
\omega =h_0. \]
Using the identity
\[ D(\xi)=d_0 \om - \om \pv (\om \pv D(\xi)),\]
one finds for all $\xi, \, \xi_+$:
\[ \om \pv (\om \pv D(\xi)) = \om \pv (\om \pv D(\xi_+)) {\rm e}^{\xi - \xi_+ },\]
therefore $ D(\xi)$, $D(\xi_+)$ and $\omega$ are coplanar. Here, it is convenient to fix
$U(\xi_+)=u_+$ in $\Omega(\om)$. Then $U(\xi) \in \Omega(\om)$ for all
$\xi$. Let us define $\zeta$ as
\begin{equation}\label{zeta-5choc}
\zeta = - \, \frac {\omega \pv (\omega \pv D_+)}{|\omega \pv (\omega
  \pv D_+)|} \, ,
\end{equation}
and set $d_+=|\omega  \pv (\omega \pv
D_+)|>0$.
We have 
\[
D(\xi)=d_0 \om + d_+ {\rm e}^{\xi - \xi_+ }\zeta,
\]
and for $\xi_-<\xi_+$:
\[ H(\xi_-)=H_++\left( \int_{\xi_-}^{\xi_+} \lambda(d_0 \omega + d_+{\rm
  e}^{\xi - \xi_+ }\zeta)d_+{\rm e}^{\xi - \xi_+ }d \xi \right) \,  \omega \pv \zeta.
\]
Going into details, we remark that $\omega$ being fixed, for $D \in \R^6$,  if $D \cdot \omega=d_0$ and $|\omega \pv (\omega \pv D)|=d$, then $|D|=\sqrt{d_0^2+d^2}$ and by (\ref{defP}), $\lambda(D)$ is a function of $d$ and $d_0$ only, that we still denote $\lambda$:
\[
\lambda^2(d,d_0)=c^2 \frac{1+\epsilon_r\left(\nn +2 \left( \displaystyle \frac{d_0}{\epsilon_0 (1+\epsilon_r \nn)}\right)^2\right)}{(1+\epsilon_r \nn)(1+3 \epsilon_r \nn)}, \quad |E|=p \left(\sqrt{d_0^2+d^2}\right).
\]
Therefore, denoting $d_-=d_+{\rm e}^{\xi_- - \xi_+ }$ and
\begin{equation}\label{def rarefct}
\rare(d_1,d_2,d_0)=\int_{d_1}^{d_2} \lambda(s,d_0)ds,
\quad 0\leq d_1 \leq d_2 \, ,
\end{equation}
we have : 
\[
H(\xi_-)=H_++\rare(d_-,d_+,d_0)\,  \omega \pv \zeta.
\]
The function $\psi=\lambda_2\circ U$, is strictly increasing by
proposition \ref{vnl}. For $\xi_- < \xi_+$, $u_\pm=U(\xi_\pm)$, one
defines
\begin{equation}\label{rare3D}
u(y)=\left\{ \begin{array}{lr}
u_- & if \quad y \leq \lambda_2(u_-), \\
U(\psi^{-1}(y))& if \quad \lambda_2(u_-)\leq y \leq \lambda_2(u_+),
\\
u_+& if \quad y \geq \lambda_2(u_+).
\end{array}
\right.
\end{equation}
Then $u(\frac{x \cdot \om}{t})$ is a centred rarefaction wave for
system (\ref{Kerr3D}), see \cite{bressan}, \cite{serre}. 

Moreover, if $\xi_- \rightarrow - \infty$, then $d_- \rightarrow 0$,
and $H(\xi_-)$ owns also a limit, so that we can
extend the definition to left states $u_-\not \in \Omega(\om)$. As a consequence the following proposition holds:
\begin{prop}\label{prop-rare-3D}
Let $u_+=(D_+,H_+)\in \Omega(\om)$ be a given right state. Using
notation (\ref{zeta-5choc}):
\[ D_+=d_0 \om + d_+ \zeta, \qquad d_0 \in \R, \quad d_+ >0.\]
For $ 0 \leq d_- \leq d_+$, let $u_-$ be defined by
\[ 
D_-=d_0 \om + d_-\zeta, \qquad H_-=H_++ \rare(d_-,d_+,d_0)\,  \omega \pv \zeta.
\]
Then $u_-$ and $u_+$ are connected by a 2-rarefaction wave.
\end{prop}
By symmetry we deduce the 5-rarefaction waves:
\begin{prop}\label{prop-rare-3D-5}
Let $u_-=(D_-,H_-)\in \Omega(\om)$ be a given left state. Using
notation (\ref{zeta-5}):
\[ D_-=d_0 \om + d_- \zeta, \qquad d_0 \in \R, \quad d_- >0.\]
For $ 0 \leq d_+ \leq d_-$, let $u_+$ be defined by
\[ 
D_+=d_0 \om + d_+\zeta, \qquad H_+=H_-- \rare(d_+,d_-,d_0)\,  \omega \pv \zeta.
\]
Then $u_-$ and $u_+$ are connected by a 5-rarefaction wave.
\end{prop}
\subsection{Wave curves}\label{wave-full-kerr}
As a conclusion to this paragraph, we define the 2 and 5 wave
curves. Let $\phi$ the function defined for $d_1\geq 0$, $d_2 \geq 0$ and $d_0\in \R$ by
\begin{equation}\label{w-fct-phi}
\phi(d_1,d_2,d_0)=\left\{
\begin{array}{rr}
\choc(d_1,d_2,d_0)& {\rm if} \quad d_2 \leq d_1\, , \\
-\rare(d_1,d_2,d_0)& {\rm if} \quad d_1 \leq d_2 \, .
\end{array}
\right.
\end{equation}
\begin{prop}\label{decphi}
$\phi$ is a decreasing $C^1$ function with respect to $d_2$ and for
all $d>0$, $d_0 \in \R$:
\begin{equation}\label{phienzero}
\phi(d,0,d_0)=\frac{ c d}{\sqrt{1+\epsilon_r p^2(\sqrt{d_0^2+d^2})}}\, , \quad \lim_{d_2 \rightarrow +\infty} \phi(d,d_2,d_0)=-\infty .
\end{equation}
\end{prop}
\begin{proof} We have 
\[ 
\partial_2 \phi(d_1,d_2,d_0)= 
\left\{ 
\begin{array}{lr}
\displaystyle \frac{f(d_2,d_0) - f(d_1,d_0)-\partial_1
    f(d_2,d_0)(d_1-d_2)}{2 \choc(d_1,d_2,d_0)}& if \quad
d_1> d_2\, , \\ \\
-\lambda ( d_2,d_0) & if \quad
d_1 < d_2\, .
\end{array}
\right.
\]
In \cite{dabh-CMS}, we have proved that $f$ is a twice differentiable concave increasing function with respect to $d_2
\geq 0$ with
\begin{equation}\label{deriv-f}
\partial_1 f(d_2,d_0)= \lambda^2(d_2,d_0).
\end{equation}
Therefore one obtains that 
\[
\lim_{d_2 \rightarrow d_1^\pm} \partial_2 \phi(d_1,d_2,d_0)=-\lambda ( d_1,d_0 ) ,
\]
which proves that $\phi$ is $C^1$, and $\partial_2 \phi(d_1,d_2,d_0)<0$ for all $d_2$. 

The first equality in (\ref{phienzero}) is immediate. To prove the second one, we first remark that 
\[
\lambda^2(s,d_0)\geq \frac{c^2}{1+\epsilon_r p^2(\sqrt{s^2+d_0^2})}
\]
and we perform the change of variable $w=p^2(\sqrt{s^2+d_0^2})$:
\[
\rare(d,d_2,d_0) \geq \frac{\epsilon_0^2}{2}\int_{w(d)}^{w(d_2)} \frac{c}{\sqrt{1+\epsilon_r w}}\frac{1+4\epsilon_r w + 3 \epsilon_r^2 w^2}{\sqrt{\epsilon_0^2 w(1+\epsilon_r w)^2 - d_0^2} }\, dw.
\]
When $d_2$ tends to $+ \infty$, so does $w(d_2)$, hence the result.
\end{proof} 
If $u_- \not = u_+$ are connected by a Lax k-shock or a k-rarefaction wave, $u_-$ and $u_+$ are said to be connected by a k-wave. In such a case, $D_- \not = D_+$ and $D_- \cdot \om = D_+ \cdot \om $. Moreover $\omega \pv (\om \pv D_-)$  and $\om \pv (\om \pv D_+)$ are colinear. 
\begin{prop}\label{lesondes}
Let us consider $u_-$ and $u_+$ such that $D_- \not = D_+$ and $D_- \cdot \om = D_+ \cdot \om =d_0$. If $\om \pv D_+ \not = 0$, we define $\zeta$ by (\ref{zeta-5choc}). Else, $\om \pv D_- \not=0$ and we define $\zeta$ by (\ref{zeta-5}). 

$u_-$ and $u_+$ are connected by a 2-wave if there exist two distinct nonnegative real numbers $d_-$, $d_+$ such that
\begin{equation}\label{2-wave}
D_\pm=d_0 \om + d_\pm \zeta, \quad H_+ = H_- + \phi (d_-, d_+, d_0) \om \pv \zeta \, .
\end{equation}
$u_-$ and $u_+$ are connected by a 5-wave if there exist two distinct nonnegative real numbers $d_-$, $d_+$ such that
\begin{equation}\label{5-wave}
D_\pm=d_0 \om + d_\pm \zeta, \quad H_+ = H_- + \phi (d_+, d_-, d_0) \om \pv \zeta \, .
\end{equation}
\end{prop}
\subsection{Solution of the Riemann
  problem}\label{sec-sol-3D}
Suppose that $u_\pm=(D_\pm,H_\pm)$ and $\omega\in R^3$,
$|\omega|=1$, are given. 
We look for intermediate states $u_1$, $u_*$, $u_{**}$, $u_2$ such that:
\begin{itemize}
\item $u_-$ and $u_1$ are connected by a 1-contact
  discontinuity,
\item $u_1$ and $u_*$ are connected by a 2-wave, 
\item $u_*$ and $u_{**}$ are connected by a stationary contact discontinuity, 
\item $u_{**}$ and $u_2$ are connected by a 5-wave, 
\item $u_2$ and $u_+$ are connected by a 6-contact discontinuity.
\end{itemize}
In the following we shall denote $d_0^\pm=D_\pm \cdot \om$.
\subsubsection{Necessary conditions.} 
Suppose that a solution exists. For the contact discontinuities 1 and 6, the following conditions have to be fulfilled:
\begin{equation}\label{ddc-total}
\left\{
\begin{array}{l}
D_1 \cdot \omega=D_- \cdot \omega =d_0^-,\\
D_2 \cdot \omega=D_+ \cdot \omega =d_0^+,\\
|D_1|=|D_-|,\\
|D_2|=|D_+|,
\end{array}
\right.
\end{equation}
\begin{equation}\label{ddc-total-H}
\left\{
\begin{array}{l}
H_1-H_-= \sigma_- \,  \omega \pv (D_1-D_-), \\
H_+-H_2= \sigma_+ \,  \omega \pv (D_+-D_2),
\end{array}
\right.
\end{equation}
with
\[ 
\sigma_-=\lambda_1(u_-)=\lambda_1(u_1), \quad \sigma_+=\lambda_6(u_+)=\lambda_6(u_2),
\]
that is
\begin{equation}\label{sigmapm}
\sigma_\pm=\pm c \left(1+ \epsilon_r |E_\pm| \right)^{-\ud} .
\end{equation} 
For the 2 and 5 waves we know that $D_1$, $D_*$, $\omega$
are coplanar and $D_2$, $D_{**}$, $\omega$
are coplanar. Moreover $[D] \cdot \om=0$. There exist unitary vectors $\zeta_1$, $\zeta_2$, orthogonal to
$\omega$ such that
\[ D_1=d_0 ^-\omega + d_1 \zeta_1, \quad D_*=d_0^- \omega + d_{*}
\zeta_1\]
and
\[ D_2=d_0^+ \omega + d_2 \zeta_2, \quad   D_{**}=d_0^+ \omega +
d_{**} \zeta_2\]
and $d_1$, $d_*$, $d_{**}$, $d_2$ are non negative.

The stationary contact discontinuity is defined by conditions
(\ref{rhstat}). One has 
\[ E_*=e_0^* \om + e_* \zeta_1, \quad E_{**}=e_0^{**} \om + e_{**}
\zeta_2\, ,\]
where 
\[ 
e_*=\frac{d_*}{\epsilon_0(1+\epsilon_r p^2(|D_*|))}, \quad
e_{**}=\frac{d_{**}}{\epsilon_0(1+\epsilon_r p^2(|D_{**}|))}\, .
\]
Therefore $e_* \, \om \pv \zeta_1=e_{**}\, \om \pv \zeta_2$. Hence
either $e_* =e_{**}=0$ or those quantities are both positive and
$\zeta_1=\zeta_2$. The first case occurs if and only if $\om \pv
D_*=\om \pv D_{**}=0$. In the second case we have $e_*=e_{**}$, which
also reads as 
\begin{equation}\label{ddcstat-bis}
f(d_*,d_0^-)=f(d_{**},d_0^+).
\end{equation}
\\[12pt]
{\bf First case:} $\om \pv D_*=\om \pv D_{**}=0$.

In that case, $D_*=d_0^- \om$, $D_{**}=d_0^+ \om$. $u_1$ and $u_*$ are
the left and right states of a 2-shock propagating with speed
\[
\sigma_2=-\sqrt{\frac{f(d_1,d_0^-)-f(0,d_0^-)}{d_1}}=\sigma_- \, .
\]
In the same way, $u_{**}$ and $u_2$ are the left and right states of a
5-shock propagating with speed $\sigma_+$. Consequently the contact
discontinuities merge with the shocks, see Figure \ref{dege-cour}.
\begin{figure}
\centering
\includegraphics[width=0.5\linewidth,angle=0]{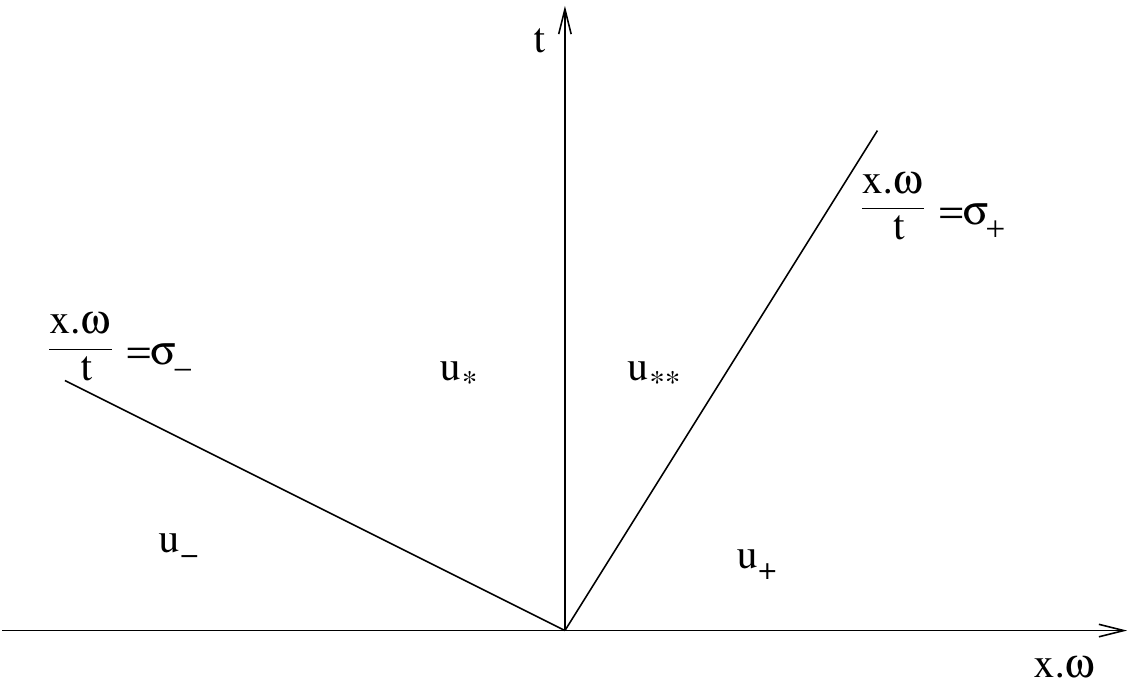}
\caption{Case $\om \pv D_*=\om \pv D_{**}=0$.}
\label{dege-cour}
\end{figure}
We have the following
relations:
\[
\left\{
\begin{array}{l}
H_1-H_-=\sigma_- \omega \pv (D_1-D_-),\\
H_*-H_1=-\sigma_- \omega \pv D_1\\
H_2-H_{**}=\sigma_+ \omega \pv D_2\\
H_+-H_2=\sigma_+ \omega \pv (D_+-D_2).
\end{array}
\right.
\]
Let us denote
\begin{equation}\label{def-V}
V=\om \pv \left( H_+-H_--\omega \pv (\sigma_+ D_+ -\sigma_- D_-)\right).
\end{equation}
Using the second relation of (\ref{rhstat}):
\begin{equation}\label{patho}
V=0,
\end{equation}
and
\begin{equation}\label{patho-H}
H_*=H_--\omega \pv \sigma_- D_-  \, , \quad H_{**}=H_+-\omega \pv \sigma_+ D_+  \, .
\end{equation}
If $D_- \pv \om=0$ then $u_-=u_*$. Else one has $D_-=d_0^- \om + d_-
\zeta$ with $\zeta$ defined by (\ref{zeta-5}) so
\[ H_*= H_- - \sigma_-  \, d_- \, \om \pv \zeta=H_- + \phi(d_-,0,d_0^-)
\om \pv \zeta. \]
This proves that $u_-$ and $u_*$ are connected by a Lax 2-shock. 

In the same way, if  $D_+\pv \om=0$ then $u_+=u_{**}$, else $u_+$ and
$u_{**}$ are connected by a Lax 5-shock. 
\\[12pt]
{\bf Second case:} $D_* \pv \om \not=0$ and $D_{**}\pv \om \not=0$.

In this case, $\zeta_1=\zeta_2=\zeta$ and 
\begin{equation}\label{gene-3D-D-12}
D_1=d_0^- \omega + d_1 \zeta, \quad D_2=d_0^+ \omega + d_2 \zeta  ,
\end{equation}
\begin{equation}\label{gene-3D-D-star}
D_*=d_0^- \omega + d_* \zeta, \quad D_{**}=d_0^+ \omega + d_{**} \zeta, 
\end{equation}
with $d_1\geq 0$, $d_*>0$, $d_{**}>0$, $d_2\geq 0$.  Let us denote
\begin{equation}\label{dh}
d=D \cdot \zeta, \quad h=H \cdot (\om \pv \zeta).
\end{equation} 
By
(\ref{ddc-total}-\ref{ddc-total-H}):
\begin{equation}\label{d1d2}
d_1= | \om \pv (\om \pv D_-)|, \quad  d_2= | \om \pv (\om \pv D_+)|, 
\end{equation}
and
\[
h_1= h_- + \sigma_- (d_1-d_-), \quad h_2 = h_+ + \sigma_+(d_2-d_+). 
\]
By proposition \ref{lesondes}, for the 2-wave curve connecting $u_1$ and $u_*$:
\begin{equation}\label{2W}
H_* - H_1= \phi (d_1,d_*,d_0^-) \om \pv \zeta \, .
\end{equation}
In the same way:
\begin{equation}\label{5W}
H_2-H_{**}  = \phi (d_2,d_{**},d_0^+) \om \pv \zeta \, .
\end{equation}
By (\ref{rhstat}), $h_*=h_{**}$ and 
\begin{equation}\label{2-5-ondes}
h_*=h_1+\phi(d_1,d_*,d_0^-) =h_2-\phi(d_2,d_{**},d_0^+).
\end{equation}
Therefore, using (\ref{ddcstat-bis}), we see that $d_*$ and $d_{**}$ are solution of the two by two system:
\begin{equation}\label{223D}
\left\{
\begin{array}{l}
f(d_*,d_0^-)=f(d_{**},d_0^+)\, ,\\ \\
h_1+\phi(d_1,d_*,d_0^-) =h_2-\phi(d_2,d_{**},d_0^+).
\end{array}
\right.
\end{equation}
As $\phi$ is decreasing and $d_*$,  $d_{**}$ are positive:
\begin{equation}\label{CNsigne}
 \phi(d_1,d_*,d_0^-) +\phi(d_2,d_{**},d_0^+) < \phi(d_1,0,d_0^-) +\phi(d_2,0,d_0^+)= \sigma_+
 d_2-\sigma_- d_1 .
\end{equation}
This inequality is useful to determine $\zeta$. As a matter of fact, using (\ref{ddc-total-H}), we have also
\[
\left\{
\begin{array}{l}
H_*=H_- + \sigma_- \om \pv (D_1-D_-) + \phi (d_1,d_*,d_0^-) \om \pv \zeta ,\\
H_{**}=H_+ - \sigma_+ \om \pv (D_+-D_2)-\phi (d_2,d_{**},d_0^+) \om \pv \zeta \, .
\end{array}
\right.
\]
Again by (\ref{rhstat}), using notation (\ref{def-V}):
\[ V= \left( \sigma_+ d_2 -\sigma_- d_1 -\phi (d_1,d_*,d_0^-) -\phi (d_2,d_{**},d_0^+) \right) \zeta .\]
Therefore $V \not= 0$ and 
\begin{equation}\label{def-zeta}
\zeta=\frac{V}{|V|}.
\end{equation}
We sum up the results in the following proposition.
\begin{prop}\label{3D-CN}
Consider $u_-$, $u_+$ such that the Riemann problem for system
(\ref{Kerr3D}) has a solution which is a superposition of simple
waves. Let $V$ be defined by (\ref{def-V}). Then only the following two cases occur:

1) $V=0$, $u_-$ and $u_*$ are connected by a Lax 2-shock propagating
with velocity $\sigma_-$,
$u_+$ and $u_{**}$ are connected by a Lax 5-shock propagating
with velocity $\sigma_+$, 
$D_*=(D_- \cdot \om)\om$, $D_{**}=(D_+ \cdot \om)\om$, $H_*$ and
$H_{**}$ are given by (\ref{patho-H}).

2) $V \not=0$, $\zeta$ is defined by (\ref{def-zeta}), $u_1$ and $u_2$
are determined by conditions (\ref{ddc-total-H}), (\ref{sigmapm}), (\ref{gene-3D-D-12}), (\ref{d1d2}) and
$u_*$, $u_{**}$ are determined by (\ref{gene-3D-D-star}), (\ref{2W}-\ref{5W}) and the solution of system (\ref{223D}).
\end{prop}
\subsubsection{Sufficient conditions}
Consider an initial Riemann data. We consider two cases according as
$V=0$ or not.
\\[12pt]
{\bf First case: $V=0$.} We define $D_*=(D_- \cdot \om)\om$,
$D_{**}=(D_+ \cdot \om)\om$, $H_*$ and $H_{**}$ by (\ref{patho-H}). It
is easy to see that $u_-$ and $u_*$ are connected by a Lax 2-shock,
$u_*$ and $u_{**} $ are connected by a stationary contact
discontinuity, $u_{**}$ and $u_+$ are connected by a Lax 5-shock, so
we have constructed the solution of the problem.
\\[12pt]
{\bf Second case: $V \not =0$.} 
We define $\zeta$ by (\ref{def-zeta}). Then we set
\[ 
D_\pm=d_0^\pm \om + d_\pm \zeta + d^\prime_\pm \om \pv \zeta, \quad H_\pm=h_0^\pm \om + h_\pm ^\prime \zeta + h_\pm \om \pv \zeta,
\]
so that
\[
V=(-h_+ + h_- +\sigma_+ d_+ - \sigma_- d_-)\zeta +(h_+^\prime - h_-^\prime + \sigma_+ d_+^\prime - \sigma_- d_-^\prime) \om \pv \zeta.
\]
Hence we can state:
\begin{lemma}\label{toto}
The two following properties hold:
\begin{equation}\label{toto1}
-h_+ + h_- +\sigma_+ d_+ - \sigma_- d_- >0,
\end{equation}
\begin{equation}\label{toto2}
h_+^\prime - h_-^\prime + \sigma_+ d_+^\prime - \sigma_- d_-^\prime =0.
\end{equation}
\end{lemma}
Using notations (\ref{sigmapm}), (\ref{dh}), we define $u_1$, $u_2$ by (\ref{gene-3D-D-12}), (\ref{d1d2}) and (\ref{ddc-total-H}). Clearly, $u_-$ and $u_1$ are connected by a 1-contact discontinuity, $u_+$ and $u_2$ are connected by a 6-contact discontinuity. Then we solve system (\ref{223D}):
\begin{lemma}\label{fctimpl} 
The system (\ref{223D}) has a unique solution $(d_*,d_{**}) \in \R_+^2$.
\end{lemma}
\begin{proof}
The values of $d_0^\pm$, $d_1$, $d_2$, $h_1$, $h_2$ are fixed. Denoting $f_\pm=f(\cdot,d_0^\pm)$, we know that $f_+$ and $f_-$ are increasing, $C^1$-diffeomorphisms from $\R$ to $\R$ such that $f_-(0)=f_+(0)=0$, see \cite{dabh-CMS}. Hence we can define $G=f_+^{-1} \circ f_-$, which is a $C^1$ increasing one-to-one function such that $G(0)=0$. We only need to define $G$ on $\R_+$: $G( \R_+)=\R_+$. 

Solving system (\ref{223D}) is equivalent to find $d_*\in \R_+$ such that
\begin{equation}\label{toto100}
h_2-h_1= \phi(d_1,d_*,d_0^-)+\phi(d_2,G(d_*),d_0^+).
\end{equation}
We have
\[
h_2-h_1=h_+ - h_- -\sigma_+ d_++ \sigma_- d_- + \phi(d_1,0,d_0^-)+\phi(d_2,0,d_0^+).
\]
By using (\ref{toto1}), we deduce 
\[
h_2 - h_1 < \phi(d_1,0,d_0^-)+\phi(d_2,0,d_0^+).
\]
Hence by proposition \ref{decphi}, the solution $d_*$ of (\ref{toto100}) exists and is unique. Setting $d_{**}=G(d_*)$, $(d_*,d_{**})$ is the unique solution of system (\ref{223D}).
\end{proof}
Let $(d_* , d_{**})$ be the solution of system  (\ref{223D}). We define $u_*$, $u_{**}$ by (\ref{gene-3D-D-star}), (\ref{2W}), (\ref{5W}). By construction, $u_1$ and $u_*$ are connected by a 1-wave, $u_2$ and $u_{**}$ are connected by a 5-wave. 

It remains to verify that $u_*$ and $u_{**}$ are connected by a stationary contact discontinuity. First it is easy to see that $[H \pv \om]=0$ if and only if $|V|=-h_+ + h_- +\sigma_+ d_+ - \sigma_- d_-$, which a consequence of lemma \ref{toto}. Moreover $[E \pv \om]=0$ if and only if (\ref{ddcstat-bis}) is satisfied, which is true because $(d_*,d_{**})$ is solution of system (\ref{223D}). 

We sum up the results in the following theorem:
\begin{theorem}\label{exis-sol}
Let $u_-$, $u_+$ be a Riemann data for system (\ref{Kerr3D}) in the direction $\omega$. The Riemann problem has a unique solution in the class of the functions which are superpositions of simple waves as detailed at beginning of section \ref{sec-sol-3D}. Let $V$ be the vector defined in (\ref{def-V}).

If $V=0$, then the solution is the superposition of a  Lax 2-shock, a stationary contact discontinuity and a Lax 5-shock.

If $V\not = 0$, then the solution is the superposition of a  1-contact discontinuity, a 2-wave (Lax shock or rarefaction), a stationary contact discontinuity, a 5-wave (Lax shock or rarefaction) and a 6-contact discontinuity.

In each case, the solution is constructed as in proposition \ref{3D-CN}.
\end{theorem}
\subsection{The Transverse Magnetic case: $6 \times 6$ viewpoint}\label{66vpoint}

Let us detail the solution of the Riemann problem (\ref{Kerr3D})(\ref{riem-data}) in that case, that is $\om=(\om_1,\om_2,0)$, and $u_\pm$ is such that $D_\pm=(D_{1,\pm},D_{2,\pm},0)$, $H_\pm=(0,0,H_{3,\pm})$. Then the vector $V$ defined in (\ref{def-V}) takes the form:
\[ 
V=v \left(
\begin{array}{c}
\om_2 \\
-\om_1 \\
0
\end{array}
\right), \; v=H_{3,+}-H_{3,-}-\om_1 (\sigma_+ D_{2,+}-\sigma_- D_{2,-})+\om_2(\sigma_+ D_{1,+}-\sigma_- D_{1,-}).
\]
If $v=0$, the intermediate states $u_{*}$, $u_{**}$ are clearly tranverse magnetic: the electric field is colinear to $\om$ and the magnetic one is given by (\ref{patho-H}).

Else, the vector $\zeta$ defined by (\ref{def-zeta}) reads as
\[
\displaystyle \frac{v}{|v|} \left(
\begin{array}{c}
\om_2 \\
-\om_1 \\
0
\end{array}
\right).
\]
Therefore, all the intermediate states are TM, see formulas (\ref{ddc-total-H}), (\ref{gene-3D-D-12}-\ref{gene-3D-D-star}), (\ref{2W}-\ref{5W}).

Hence, the solution of the Riemann problem (\ref{Kerr3D})(\ref{riem-data}) with Tansverse Magnetic data is Transverse Magnetic. Moreover it is easy to see that if the Riemann data are divergence free, so is the solution.

{\bf Particular case of the p-system}.  
Here we consider the Riemann problem (\ref{Kerr3D})(\ref{riem-data})  with
$\omega=(1,0,0)$ and $u_\pm$ is such that $D_\pm=(0,D_{2,\pm},0)$, $H_\pm=(0,0,H_{3,\pm})$. This field is divergence free. Then 
\[ 
V=v \left(
\begin{array}{c}
0\\
-1 \\
0
\end{array}
\right), \quad v=H_{3,+}-H_{3,-}-(\sigma_+ D_{2,+}-\sigma_- D_{2,-}).
\]
If $v=0$, we find $D_{*}=D_{**}=0$ and $H_{*}=H_{**}=(0,0,H_3)$ with 
\[
H_3=H_{3,+}-\sigma_+ D_{2,+}=H_{3,-}-\sigma_- D_{2,-}.
\]

Else, the vector $\zeta$ defined by (\ref{def-zeta}) reads as
\[
\zeta=\displaystyle \frac{v}{|v|} \left(
\begin{array}{c}
0\\
-1 \\
0
\end{array}
\right).
\]
To avoid confusion we denote $u_1=(D^{(1)},H^{(1)})$, $u_2=(D^{(2)},H^{(2)})$ the intermediate states 1 and 2. We have $D^{(1)}=d_1 \zeta$, $D^{(2)}=d_2 \zeta$, $D_\pm=d_\pm \zeta$ with $d_\pm=-\displaystyle \frac{v}{|v|} D_{2,\pm}$, $d_1=|d_-|$, $d_2=|d_+|$.  Hence by (\ref{ddc-total-H}), $H^{(1)}=(0,0,H_3^{(1)})$ and $H^{(2)}=(0,0,H_3^{(2)})$. Moreover $D_*=D_{**}=d_* \zeta$ and by (\ref{2W}-\ref{5W}), $H_*=H_{**}=(0,0,H_{3,*})$. Thus, the solution of the $6 \times 6$ Riemann problem is a solution of the $2 \times 2$ p-system (\ref{kerr}).

Let us remark that the stationary contact discontinuity is trivial but that if $d_-<0$ ({\it resp} $d_+<0$), the contact discontinuity 1 ({\it resp} 6) is not.
\section{The Riemann problem for the $3 \times 3$ Transverse Magnetic case}\label{2DTMsec}
Here, we study the reduced TM system (\ref{33TM}), which is important for the applications. 

In this section, we use two components vectors: $D=(D_1,D_2)$, $\om=(\om_1,\om_2)$, $E=\E(D)$. We denote 
\[
\om \pv D=\om_1 D_2 - \om_2 D_1, \quad \om^\perp=(-\om_2,\om_1).
\]
\subsection{Wave curves}
Following the lines of the $6 \times 6$ case, we can prove the following:
\begin{prop}\label{vpvp2DTM} 
The TM Kerr system (\ref{33TM}) is hyperbolic diagonalizable: for
all $\omega \in \R^2$, $|\omega|=1$, the
eigenvalues are given by 
\begin{equation}\label{3vp}
\lambda_{1}=-\lambda< \lambda_{3}=0 < \lambda_{3}=\lambda
\end{equation}
where
\begin{equation}\label{3vpvp}
\lambda^2=c^2 \, \frac{1+\epsilon_r(\nn +2 (E\cdot
  \omega)^2)}{(1+\epsilon_r \nn)(1+3 \epsilon_r \nn)}.
\end{equation}
If $\om \pv D \not=0$, the eigenvectors for $\lambda_{1}$ and $\lambda_{3}$ are
\[ r_1=\left(
\begin{array}{c}
(\omega \pv D) \om^\perp  \\
-\lambda(\omega \pv D)
\end{array}\right), \quad r_3=\left(
\begin{array}{c}
-(\omega \pv D)\om^\perp   \\
-\lambda(\omega \pv D)
\end{array}\right).
\]
The characteristic fields related to $\lambda_1$ and $\lambda_3$ are genuinely nonlinear in the domain
\[ \Omega(\omega)=\{ (D,H_3)\in \R^3 \; ; \; \omega \pv D \not=0  \}
\]
and for all $u\in \Omega(\omega)$
\begin{equation}\label{vnldef2DTM}
\lambda^\prime_{i}(u,\omega) \, r_{i}(u,\omega) > 0, \quad i=1,3.
\end{equation} 
\end{prop}
When we compare the $6 \times 6$ and $3 \times 3$ situations, we observe that the reduction to TM fields makes the eigenvalues related to non-stationary contact discontinuities disappear. This is easily understandable since we have seen that those waves induce a 3D rotation of the electric field, namely a rotation around the direction of $\om$.

The stationary contact discontinuities are characterized as in proposition \ref{ddc0}. 

For a fixed left state $u_-=(D_-,H_{3,-})$ we can proceed as in \cite{dabh-CMS} to determine the related Hugoniot set, that is the set of all right states $u_+$ satisfying the Rankine-Hugoniot relations, which read here:
\[
\sigma [D]= [H_3] \omega^\perp, \qquad \sigma \mu_0 [H_3]= [\omega \pv E].
\]
Therefore, non-stationary shocks are divergence free. As $D=(D\cdot \omega)\omega + (\omega \pv D) \omega^\perp$,
we have $D_- \cdot \omega = D_+ \cdot \omega $ and 
\[
\sigma [\omega \pv D]=[H_3], \qquad \sigma \mu_0 [H_3]=[\omega \pv E].
\]  
Consequently
\[
 \sigma^2=\frac{[\omega \pv E]}{\mu_0 [\omega \pv D]}=\frac{f(\omega \pv
  D_+,d_0)-f(\omega \pv D_-,d_0)}{(\omega \pv D_+) - (\omega \pv D_-)}   
\]
where $f$ is the function defined in (\ref{fonctionf}). The vector $\zeta$ defined in section \ref{66-3D} is not useful here. Instead we denote 
\[ D_-=d_0 \om + d_- \om^\perp . \]
\begin{prop}\label{hugo-TM}
Let $u_-=(D_-,H_{3,-})$ be a fixed left state. 

The set $\Hcal(u_-)$ of the states $u_+$ connected to $u_-$ by a non-stationary shock is the set of $u_+$ such that 
\[
\left\{
\begin{array}{lr}
D_+=d_0\omega + d_+ \om^\perp , & d_+ \in \R, \\
H_{3,+}=H_{3,-} + \sigma (d_+-d_-),
\end{array}
\right.
\]
and the shock speed $\sigma(u_-,u_+)=\sigma$ satisfies (\ref{sigmaTM}).
\end{prop}
Let us now study Lax entropy conditions (\ref{lax-2}). For a 1-shock, they read as
\[ -\lambda(u_+) \leq \sigma \leq -\lambda(u_-).\]
Those conditions are very different from the $6 \times 6$ case, where for a 2-shock the requirement $\lambda_1(u_-)\leq \sigma$ imposes a sign condition on $d_+$, see \cite{dabh-CMS}. Here, this sign condition no longer exists. Instead, we obtain that $\sigma <0$ and using (\ref{deriv-f}):
\begin{equation}\label{lax2DTM} 
\partial_1 f(d_-,d_0) \leq \frac{f(d_+,d_0)-f(d_-,d_0)}{d_+ - d_-} \leq \partial_1 f(d_+,d_0).
\end{equation}
If $d_-=0$, as $\partial_1 f(\cdot,d_0)$ is maximal for $d=0$, we have $d_+=0$, hence $u_-=u_+$. 

Else, if $d_->0$, ({\it resp} $d_-<0$), as $f(\cdot,d_0)$ is strictly concave ({\it resp} convex) on $\R^+$ ({\it resp} $\R^-$), the formula is true if $0 \leq |d_+| \leq |d_-|$ and $d_- d_+ \geq 0$, but this is not necessary. Thus, we have to go beyond the point where $\om \pv D =0$, and the characteristic field 1 is not genuinely nonlinear. The relevant condition in that case is Liu's entropy condition, see \cite{Liu74}, \cite{Liu76}. 
\begin{definition}\label{def-adm-liu-1D}
Let $u_-$ be a given left state and consider $u_+ \in \Hcal(u_-)$. The discontinuity is Liu-admissible if
\[ 
(E)\qquad \sigma(u_+,u_-)\leq \sigma(u,u_-), \qquad \forall u \in \Hcal(u_-), \; u \; 
\mbox{between }  \; u_-  \; \mbox{and }  \; u_+ \, .
\]
\end{definition}
\begin{prop}\label{prop-liu}{\bf Liu's 1-shocks}

Let $u_-$ be a given left state and consider $u_+ \in \Hcal(u_-)$, written as in proposition \ref{hugo-TM}, with $\sigma <0$. 

If $\om \pv D_- = 0$, the discontinuity is neither Liu-admissible, nor Lax-admissible. 

If $\om \pv D_- \not = 0$, let $d_*=d_*(d_-)$ be the unique real such that $d_- d_* <0$ and
\[ \partial_1 f(d_*,d_0) = \frac{f(d_-,d_0)-f(d_*,d_0)}{d_- - d_0}\,.
\]
The shock is Liu-admissible if and only if $d_+ \in [d_*,d_-]$. When the shock is Liu-admissible, it is also Lax-admissible.
\end{prop}
We point out the fact that $[d_*,d_-]$ is to be understood as the segment having $d_*$ and $d_-$ as extreme points. This result is proved in \cite{dabh-CMS} for the particular case of system (\ref{kerr}). The proof of proposition \ref{prop-liu} follows the same lines so we omit it. In the following, we shall denote 
\[ D_*=d_0 \om + d_*(d_-) \om^\perp, \quad H_{3,*}= H_{3,-}+ \sigma (d_*(d_-)-d_-), \quad u_*(u_-)=(D_*,H_{3,*}).
\]
The 3-shocks are deduced from the 1-shocks by symmetry as in proposition \ref{symmetry-choc}.

We compute the rarefaction waves as in the $6 \times 6$ case, using again the function $R$ defined in (\ref{def rarefct}). We give the result for the 1-rarefactions, the 3-rarefactions are deduced by symmetry.
\begin{prop}{\bf 1-rarefactions}\label{prop-rare-2DTM}

Let $u_+=(D_+,H_{3,+})\in \Omega(\om)$ be a given right state:
\[ D_+=d_0 \om + d_+ \om^\perp, \qquad d_0 \in \R, \quad d_+ \not = 0 .\]
For $ 0 \leq |d_-| \leq |d_+|$, $\;d_- \, d_+ \geq 0$, let $u_-$ be defined by
\[ 
D_-=d_0 \om + d_- \om^\perp, \qquad H_{3,-}=H_{3,+}+ \sgn(d_+)\rare(|d_-|,|d_+|,d_0)\, .
\]
Then $u_-$ and $u_+$ are connected by a 1-rarefaction wave.
\end{prop}
As a particular case, if we fix $u_-=(d_0 \om, H_{3,-})$, we can define a global 1-wave curve parametrized by $d_+ \in \R$ which consists of rarefactions only:
\[ 
D_+=d_0 \om + d_+ \om^\perp, \qquad H_{3,+}=H_{3,-}- \sgn(d_+)\rare(0,|d_+|,d_0)\, , \quad d_+ \in \R.
\]
Otherwise, if we fix $u_-$ such that $D_- =d_0  \om +d_- \om^\perp $, $\; d_- \not=0$, putting together Liu's 1-shocks and 1-rarefactions gives a wave curve which is defined for a parameter $d_+ \in [d_*(d_-), +\infty[$ if $d_- >0$, and a parameter  $d_+ \in ]-\infty,d_*(d_-)]$ if $d_- <0$. We complete this curve by using composed waves as explained in \cite{wen} and \cite{Liu76}, that is by the 1-rarefaction curve related to the left state $u_*(u_-)$. 

Finally, we define the wave function. For $d \not = 0$, $d_*(d)$ is defined as in proposition \ref{prop-liu}. For $d=0$, we set $d_*(0)=0$. Then for all $d_1$, $d_2$, $d_0$ we define $\varphi(d_1,d_2,d_0)$ as
\begin{equation}\label{wave2DTM}
\varphi(d_1,d_2,d_0)= \left\{
\begin{array}{l}
-\sgn(d_1) R(|d_1|,|d_2|,d_0) \quad {\rm if} \; |d_1| \leq |d_2|,\; d_1 d_2 \geq 0,\\ 
\sgn(d_1) S(d_1,d_2,d_0)  \quad  {\rm if} \; d_2 \in [d_*(d_1),d_1],\\ 
\sgn(d_1) \left(S(d_1,d_*(d_1),d_0)+R(d_*(d_1),|d_2|,d_0) \right) \quad {\rm else.} 
\end{array}
\right.
\end{equation}
As in proposition \ref{decphi}, we can prove that $\varphi$ is a decreasing $C^1$ function with respect to $d_2$ and for all $d_1$, $d_0 \in \R$:
\begin{equation}\label{varphienzero}
 \lim_{d_2 \rightarrow  \pm\infty} \varphi(d_1,d_2,d_0)=\mp \infty .
\end{equation}
\begin{prop}\label{lesondes2DTM}
Let us consider $u_-$ and $u_+$ such that $D_- \not = D_+$ and $D_- \cdot \om = D_+ \cdot \om =d_0$. 

$u_-$ and $u_+$ are connected by a 1-wave if there exist two distinct real numbers $d_-$, $d_+$ such that
\begin{equation}\label{2-wave-TM}
D_\pm=d_0 \om + d_\pm \om^\perp, \quad H_+ = H_- + \varphi (d_-, d_+, d_0) \, .
\end{equation}
$u_-$ and $u_+$ are connected by a 3-wave if there exist two distinct real numbers $d_-$, $d_+$ such that
\begin{equation}\label{5-wave-TM}
D_\pm=d_0 \om + d_\pm \om^\perp, \quad H_+ = H_- + \varphi (d_+, d_-, d_0) \, .
\end{equation}
\end{prop}
\subsection{Solution of the Riemann problem}\label{resol2DTM}
\label{sec-sol-2DTM}
Suppose that $u_\pm=(D_\pm,H_{3,\pm})$ and $\omega\in R^3$,
$|\omega|=1$, are given. 
We look for
intermediate states $u^{(1)}$, $u^{(2)}$ such that:
\begin{itemize}
\item $u_-$ and $u^{(1)}$ are connected by a 1-wave, 
\item $u^{(1)}$ and $u^{(2)}$ are connected by a stationary contact discontinuity, 
\item $u^{(2)}$ and $u_+$ are connected by a 3-wave.
\end{itemize}
In the following we shall denote $D_\pm =d_0^\pm \om + d_\pm \om^\perp$.
\subsubsection{Necessary conditions.} 
Suppose that a solution exists. There exist real numbers $d_1$, $d_2$ such that
\begin{equation}\label{2DTM-forme} D^{(1)}=d_0^- \om +d_1 \om^\perp, \quad D^{(2)}=d_0^+ \om +d_2 \om^\perp,\end{equation}
\begin{equation}\label{2DTM-H} H_{3}^{(1)}=H_{3,-}+\varphi(d_-,d_1,d_0^-), \quad H_{3}^{(2)}=H_{3,+}-\varphi(d_+,d_2,d_0^+),\end{equation}
\begin{equation}H_{3}^{(1)}=H_{3}^{(2)}, \quad f(d_1,d_0^-)=f(d_2,d_0^+). \nonumber\end{equation} 
Therefore, $(d_1,d_2)$ is solution of a two by two system which is similar to (\ref{223D}):
\begin{equation}\label{222DTM}
\left\{
\begin{array}{l}
f(d_1,d_0^-)=f(d_{2},d_0^+)\, ,\\ \\
H_{3,-}+\varphi(d_-,d_1,d_0^-) =H_{3,+}-\varphi(d_+,d_2,d_0^+).
\end{array}
\right.
\end{equation}
\subsubsection{Sufficient conditions.} 
\begin{lemma}\label{lemm2DTM22}
The system (\ref{222DTM}) has a unique solution $(d_1,d_2) \in \R^2$.
\end{lemma}
\begin{proof} 
The values of $d_0^\pm$, $d_\pm$, $H_{3,\pm}$ are fixed. We define $f_\pm$ as in the proof of lemma \ref{fctimpl}, and $G=f_+^{-1} \circ f_-$, which is a $C^1$ increasing one-to-one function such that $G(0)=0$. Here we need to define $G$ on $\R$, and $G(\R)=\R$.

Solving system (\ref{222DTM}) is equivalent to find $d_1$ such that
\[ 
H_{3,+}-H_{3,-} =\varphi(d_-,d_1,d_0^-)+\varphi(d_+,G(d_1),d_0^+).
\]
We end the proof by using the properties of $\varphi$.
\end{proof}
In the following theorem we sum up those considerations and we make the link between $2 \times 2$ and $3 \times 3$ solutions.
\begin{theorem}\label{exis-sol2DTM}
Let $u_-$, $u_+$ be a Riemann data for system (\ref{33TM}) in the direction $\omega$. The Riemann problem has a unique solution in the class of the functions which are superpositions of a 1-wave, a stationary contact discontinuity and a 3-wave. 

The intermediate states $u^{(1)}$ and $u^{(2)}$ are defined by (\ref{2DTM-forme}), (\ref{2DTM-H}), and the solution of system (\ref{222DTM}).

For Riemann data of the form $u_\pm=(0,D_{2,\pm},H_{3,\pm})$, the solution has the form $(0,D_2,H_3)$, the stationary contact discontinuity is trivial and $(D_2,H_3)$ is the Liu's solution of the p-system (\ref{kerr}) for data $(D_{2,\pm},H_{3,\pm})$.
\end{theorem}
\subsection{Comparison of the $6 \times 6$ solution with the $3 \times 3$ and $2 \times 2$ ones} \label{compar}
For $u=(D_1,D_2,H_3)$, we denote ${\overline u}=({\overline D},{\overline H})$, where ${\overline D}=(D_1,D_2,0)$, ${\overline H}=(0,0,H_3)$.

As observed in paragraph \ref{66vpoint}, if the Riemann data are Transverse Magnetic, so is the solution ${\overline u}$ of system (\ref{Kerr3D}), and the related $u$ is a weak solution of the $3 \times 3$ system (\ref{33TM}). But if the non stationary contact discontinuities are not trivial for ${\overline u}$, $u$ is not the Liu's solution of (\ref{33TM}). 

For example, we can find a non trivial Tranverse Magnetic 6-contact discontinuity for system (\ref{Kerr3D}). We choose $\om=(1,0,0)$,
\[ 
{\overline D_-}=\left( \begin{array}{c}
0\\ D_{2,-} \\ 0 
\end{array}
\right), \quad {\overline H_-}=\left( \begin{array}{c}
0\\0\\H_{3,-}
\end{array}
\right), \quad
{\overline D_+}=-{\overline D_-},
\]
and, $\sigma_+$ being defined by (\ref{sigmapm}):
\[ 
{\overline H_+}={\overline H_-}+\sigma_+ \om \pv ({\overline D_+}-{\overline D_-})=\left( \begin{array}{c}
0\\ 0\\ H_{3,-}+\sigma_+ (D_{2,+} -D_{2,-} )
\end{array}
\right).
\]
The solution of the $3 \times 3$ Riemann problem for system (\ref{33TM}) with data $u_\pm=(0,D_{2,\pm},H_{3,\pm})$ cannot be such a contact discontinuity. The solution consists of a 1-wave and a 3-wave. Such solutions are compared in Figures \ref{xm33}, see section \ref{num-sect} for the numerical details. Here, the 1-wave is a rarefaction, while the 3-wave is composed by a shock connecting $u_+$ and $u^{*}(u_+)$, and a rarefaction connecting $u^{*}(u_+)$ and $u^{(2)}=u^{(1)}$. 

Consequently one faces two distinct solutions of the problem. This is not contrary to known results. In particular, we point out the fact that, as usual for such problems, uniqueness in theorems \ref{exis-sol} and \ref{exis-sol2DTM} holds only in a definite class of solutions. 

In order to choose the physical solution, we study the electromagnetic energy of each of them. For the reduced case (\ref{33TM}), still denoting $E=\E(D)$, the energy density reads as (\cite{c-h2}):
\[ \eta(D,H_3)=\Ecal(D)+\ud \mu_0 H_3^2, \quad \Ecal(D)=\epsilon_0 (|E|^2+\frac{3 \epsilon_r}{2}|E|^4).\] 
Actually $\eta$ is a mathematical entropy for Kerr system, with entropy flux
\[ Q(D,H_3)=H_3(E_2,-E_1 ).\]
As well known, contact discontinuities and rarefactions preserve entropy, see \cite{serre} for example. Let us study what happens for Liu's shocks. 

A shock $(\sigma, u_-,u_+)$ is entropy dissipative if 
\[
\partial_t \eta(u) + {\rm div} Q(u) \leq 0
\]
in a weak sense. This inequality also reads as
\begin{equation}\label{entrdiss}
-\sigma [\eta(D,H_3)] + [ H_3 \om \pv \E(D)]\leq 0.
\end{equation} 
\begin{theorem}\label{dissi2DTMtheo}{\bf Entropy dissipation for Liu's shocks.}

Let $(\sigma,u_-,u_+,-)$ be a Liu's shock. The entropy dissipation inequality (\ref{entrdiss}) holds. 

In the particular case $D_\pm \cdot \om=0$, denoting $e=p(d)$, $D=d \om^\perp$, the amount of entropy dissipation is
\begin{equation}\label{dissi-Liu}
-\sigma [\eta(D,H_3)]+[H_3 e]=-\frac{c\;  \epsilon_0 \epsilon_r}{4\sqrt{1+\epsilon_r(e_+^2+e_+ e_- +e_-^2})}  [e]^2 \;\left|  [e^2]\right| \leq  0.
\end{equation}
\end{theorem}
\begin{proof} We write the proof for a 1-shock with $D_- \cdot \om^\perp > 0$, the other cases are similar. The Liu's 1-shock curve for given $u_-$ is parametrized by $d \in [d_*(d_-),d_-]$ as
\begin{equation}\label{defu2DTM}
u(d)=\left(
\begin{array}{c}
d_0 \om + d \om^\perp \\
H_{3,-} + \sigma(d,d_-)(d-d_-)
\end{array}
\right), \quad \sigma < 0, \quad \sigma^2=\frac{f(d,d_0)-f(d_-,d_0)}{d-d_-}.
\end{equation} 
For such a $u$, $\om \pv \E(D)=\mu_0 f(d,d_0)$, so that (\ref{entrdiss}) reads as
\[
-\sigma (\eta(u(d)) -\eta(u(d_-)))+\mu_0 f(d,d_0)(H_{3,-} + \sigma(d-d_-))-\mu_0f(d_-,d_0)H_{3,-}  \leq 0.
\]
We denote $-\sigma \Dcal(d)$ the left-hand-side of this inequality. $\Dcal(d_-)=0$ and
\[
\Dcal(d)=\Ecal(d_0 \om + d \om^\perp)-\Ecal(d_0 \om + d_- \om^\perp)-\ud \mu_0(d-d_-)(f(d,d_0)+f(d_-,d_0)).
\]
Using the fact that $\Ecal^\prime(D,H_3)=(E_1,E_2,\mu_0 H_3)$, we find
\[\Dcal^\prime(d)=\ud \mu_0 (d_--d)\left(\partial_1f(d,d_0)-\frac{f(d_-,d_0)-f(d,d_0)}{d_--d}  \right).\]
The properties of $f$ and the definition of $d_*(d_-)$ allow us to conclude that $\Dcal^\prime(d) \geq 0$ for $d \in [d_*(d_-),d_-]$, and this proves the entropy dissipation property. 

In the case where $d_0=0$, $\Ecal(D)=\epsilon_0 e^2 (1+\frac{3 \epsilon_r}{2} e^2)$ and $\mu_0 f(d,0)=p(d)=e$, hence the result.\end{proof}
As Lax' shocks are also Liu's shocks, we conclude that we have found two distinct selfsimilar entropy (or energy) solutions of the Riemann problem for (\ref{33TM}). In the case of the above example, the $6 \times 6$ solution conserves the electromagnetic energy, while Liu's solution dissipates this energy by presence of a shock. Numerical experiments will bring more information about this problem, see section \ref{num-sect}.
\section{Numerical experiments}\label{num-sect}
We present one and two dimensional computations with Godunov scheme for the $6 \times 6$ Kerr system. The one-dimensional tests are concerned with comparisons to exact solutions of the Riemann problem. As a particular case, we investigate numerically the problem of the nonuniqueness of selfsimilar entropy TM solutions. 

The two-dimensional experiments are performed on a cartesian grid. We take Transverse Magnetic data $(D_1,D_2,H_3)$ but we use the $6 \times 6$ solver, see paragraph \ref{66vpoint}. The first case is concerned with a piecewise constant initial data for which one-dimensional waves remain visible. Then we study an ultrashort optical pulse proposed in \cite{ziol-jud}. 

In all cases, we also compare our results with those obtained by a Kerr-Debye relaxation scheme, see Annex.

The relative permittivity is $\epsilon_r=2.10^{-18}$. 

All the computations have been performed with a CFL number of 0.3. An important remark is that all the characteristic velocities are bounded by the light velocity $c=\sqrt{\epsilon_0 \mu_0}^{-1}$, so that we are able to fix a constant time step. All the results are obtained with a second order extension, in space by affine reconstructions with minmod limiters, in time by a second order Runge-Kutta scheme.
\subsection{One-dimensional cases}\label{one-d}
We fix the computation domain as $[-X,X]$ with $X=cT$, $T$ being the maximal time, so that if $N$ is the number of cells, then the number of time steps is $p=N/0.6$.

We first consider the Kerr system (\ref{Kerr3D}) with the following Riemann data:
\begin{equation}\label{data1D}
\begin{array}{lr}
D(x,0)=\left( \begin{array}{c}
0\\0.03\\0
\end{array}
\right) 
\quad {\rm if} \; \; x_1<0, & D(x,0)=\left( \begin{array}{c}
0.03\\0.04\\0.04
\end{array}
\right) \quad {\rm if} \; \; x_1>0, \\ \\
\mu_0 H(x,0)=\left( \begin{array}{c}
0\\0\\3
\end{array}
\right) 
\quad {\rm if} \; \; x_1<0, & \mu_0 H(x,0)=\left( \begin{array}{c}
0.001\\0\\3
\end{array}
\right) \quad {\rm if} \; \; x_1>0.
\end{array}
\end{equation}
This data is not divergence free. The solution only depends on $x=x_1$, $D_1(x,t)=D_1(x,0)$ and $H_1(x,t)=H_1(x,0)$. The scheme (\ref{int-VF-ns}) with (\ref{phigen}) reads as
\[\begin{array}{l}
D^{n+1}_{i}=D^{n}_{i}-\displaystyle \frac{\dt}{\dx}\left(
\begin{array}{c}
0 \\ H_{3,i+\ud}^n-H_{3,i-\ud}^n \\ -H_{2,i+\ud}^n+H_{2,i-\ud}^n 
\end{array}
\right), \\
H^{n+1}_{i}=H^{n}_{i}-\mu_0^{-1}\displaystyle \frac{\dt}{\dx}\left(
\begin{array}{c}
0 \\ -E_{3,i+\ud}^n+E_{3,i-\ud}^n \\ E_{2,i+\ud}^n-E_{2,i-\ud}^n 
\end{array}
\right).
\end{array}\]
As a consequence, $D_{1,i}^n=D_{1,i}^0$ and $H_{1,i}^n=H_{1,i}^0$, and the error for those components is only due to initial discretization of data. As $x=0$ is an interface between two cells, this error is zero. Hence we do not represent $D_1$ and $H_1$.

Figures \ref{xm11}-\ref{xm1} show respectively the components $(D_2,D_3)$ and $(H_2,H_3)$ at time $T=10$ femtoseconds, for 400 and 1600 cells. The exact solution consists of a 1-contact discontinuity, a 2-rarefaction, a nontrivial stationary contact discontinuity, a 5-shock and a 6-contact discontinuity. It is well retrieved by Godunov scheme. We have also tested the Kerr-Debye relaxation scheme (\ref{schemkd1D}) with (\ref{fluxkd1D}) presented in Annex. Both scheme give very close results. In Figure \ref{xm2}, $L^1$ relative errors with respect to the space step are depicted for each of them. We make the number of cell vary from 400 to 1600. The numerical order of accuracy is 0.66.
\begin{figure}
\centering
\includegraphics[width=0.6\linewidth,angle=0]{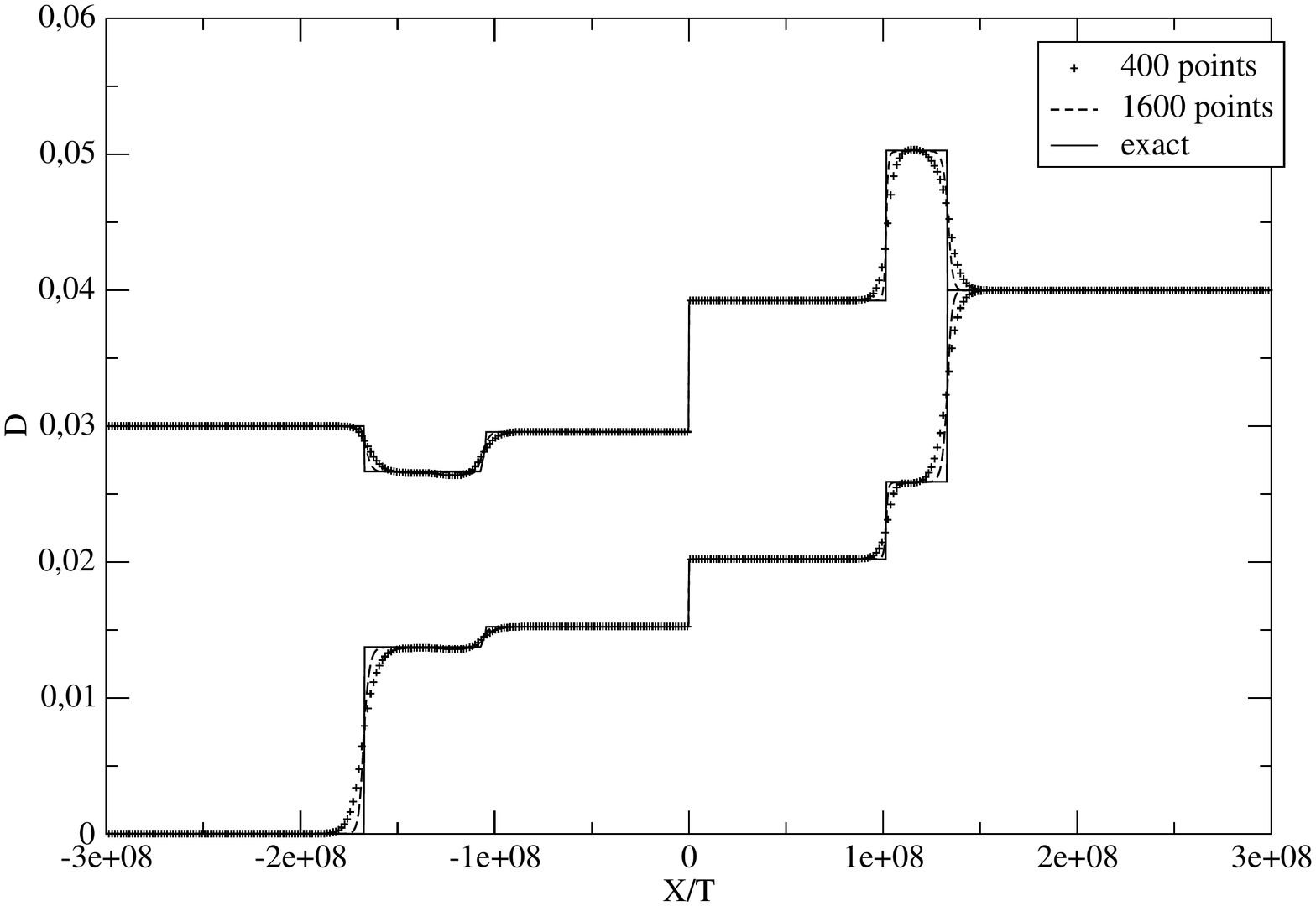}
\caption{One space dimension: $(D_2,D_3)$ for Riemann data (\ref{data1D}) afer 10 femtoseconds. Godunov scheme.}
\label{xm11}
\end{figure}
\begin{figure}
\centering
\includegraphics[width=0.6\linewidth,angle=0]{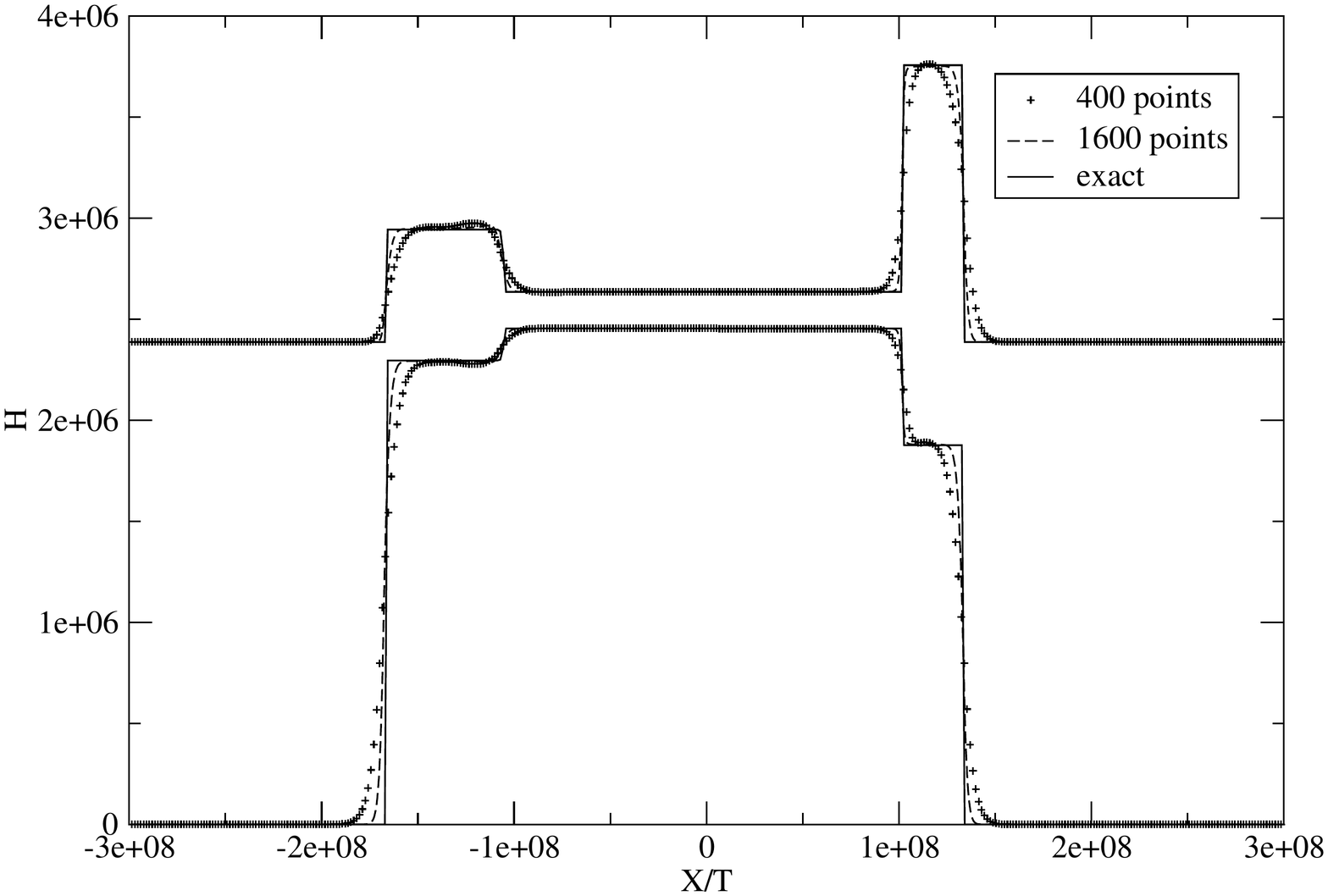}
\caption{One space dimension: $(H_2,H_3)$ for Riemann data (\ref{data1D}) afer 10 femtoseconds. Godunov scheme.}
\label{xm1}
\end{figure}
\begin{figure}
\centering
\includegraphics[width=0.6\linewidth,angle=0]{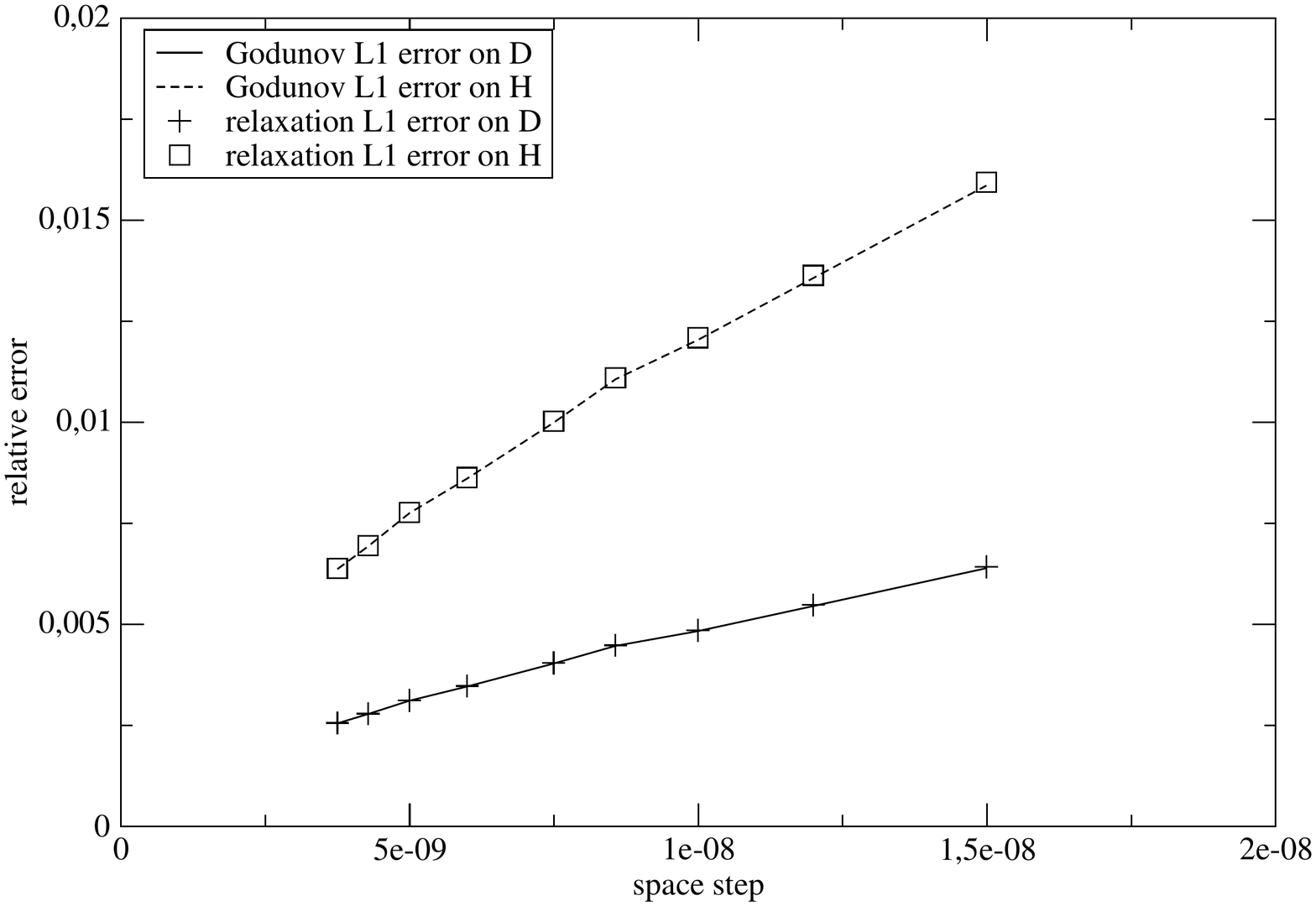}
\caption{One space dimension: relative $L^1$ errors for Godunov and Kerr-Debye relaxation schemes, for Riemann data (\ref{data1D}), from 400 to 1600 cells.}
\label{xm2}
\end{figure}
\\In a second series, we try to understand the problem of non uniqueness shown in paragraph \ref{compar}. We take a sequence $(D^{m},H^{m})$ of 6-contact discontinuities as follows: $\omega=(1,0,0)$, and for $m=1,\dots,12$: $\theta_m=\frac{m\pi}{12}$, 
\begin{equation}\label{data1D-cd}
\begin{array}{ll}
D^m(x,0)=D_-
\; {\rm if} \; x_1<0, & D^m(x,0)=D^m_+= R_m D_- \; {\rm if} \; \; x_1>0, \\ \\
H^m(x,0)=H_-  \; {\rm if} \; x_1<0, & H^m(x,0)=H^m_+ \; {\rm if} \; \; x_1>0,
\end{array}
\end{equation}
with $H^m_+=H_-+ \sigma_+ \om \pv (D^m_+ -D_-)$, $\sigma_+$ defined in (\ref{sigmapm}),
\[ D_- =\left( \begin{array}{c}
0\\0.03\\0
\end{array}
\right) , \quad \mu_0 H_-=\left( \begin{array}{c}
0\\0\\3
\end{array}
\right), \quad R_m=\left(
\begin{array}{ccc}
1& 0 & 0 \\ 
0& \displaystyle \cos \theta_m &- \displaystyle\sin \theta_m  \\ 
0&\displaystyle\sin \theta_m &\displaystyle\cos \theta_m
\end{array}
\right).
\]
When $m=12$ ($\theta_m=\pi$), we have a Transverse Magnetic field which is also a weak solution of the p-system (\ref{kerr}), and the entropy is conserved: denoting $e=p(0.03)$,
\[
-\sigma_+ [\eta(D,H_3)]+[H_3 e]=0.
\]
We have $|D_-|=|D_+^m|=0.03$, 
\[ 
|D_+^m-D_-| = |D_-|\sqrt{2(1-\cos\theta_m )} \,, 
\]
and
\[ 
|H^m_{3,+}-H_{3,-}| = \frac{c}{\sqrt{1+\epsilon_r e^2}} |D_+^m-D_-|,
\]
therefore $|u_+ - u_-|$ is an increasing function of $m$. 

Figure \ref{xm44} shows the evolution of $L^1$ relative error for $D$ and $H$ respectively. We can observe that this error increases with the rotation angle but it always converges to zero, except when $\theta_m=\pi$. As shown in Figure \ref{xm33}, in this case convergence holds to Liu's solution, which consists of a 1-rarefaction and a 2-wave composed by a 2-rarefaction and a 2-shock. The Kerr-Debye relaxation scheme gives the same results.
\begin{figure}
\centering
\includegraphics[width=0.49\linewidth,angle=0]{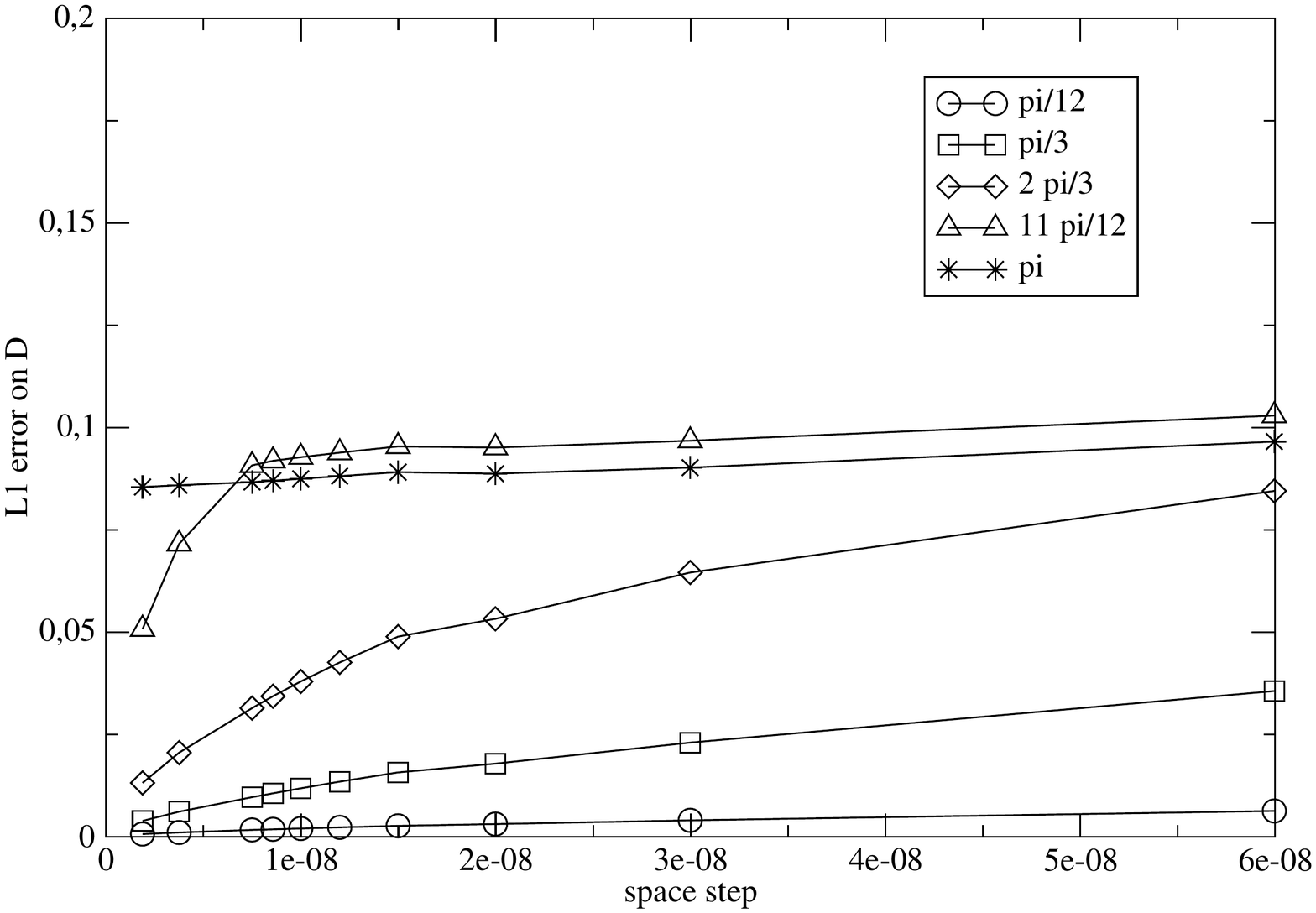}
\includegraphics[width=0.49\linewidth,angle=0]{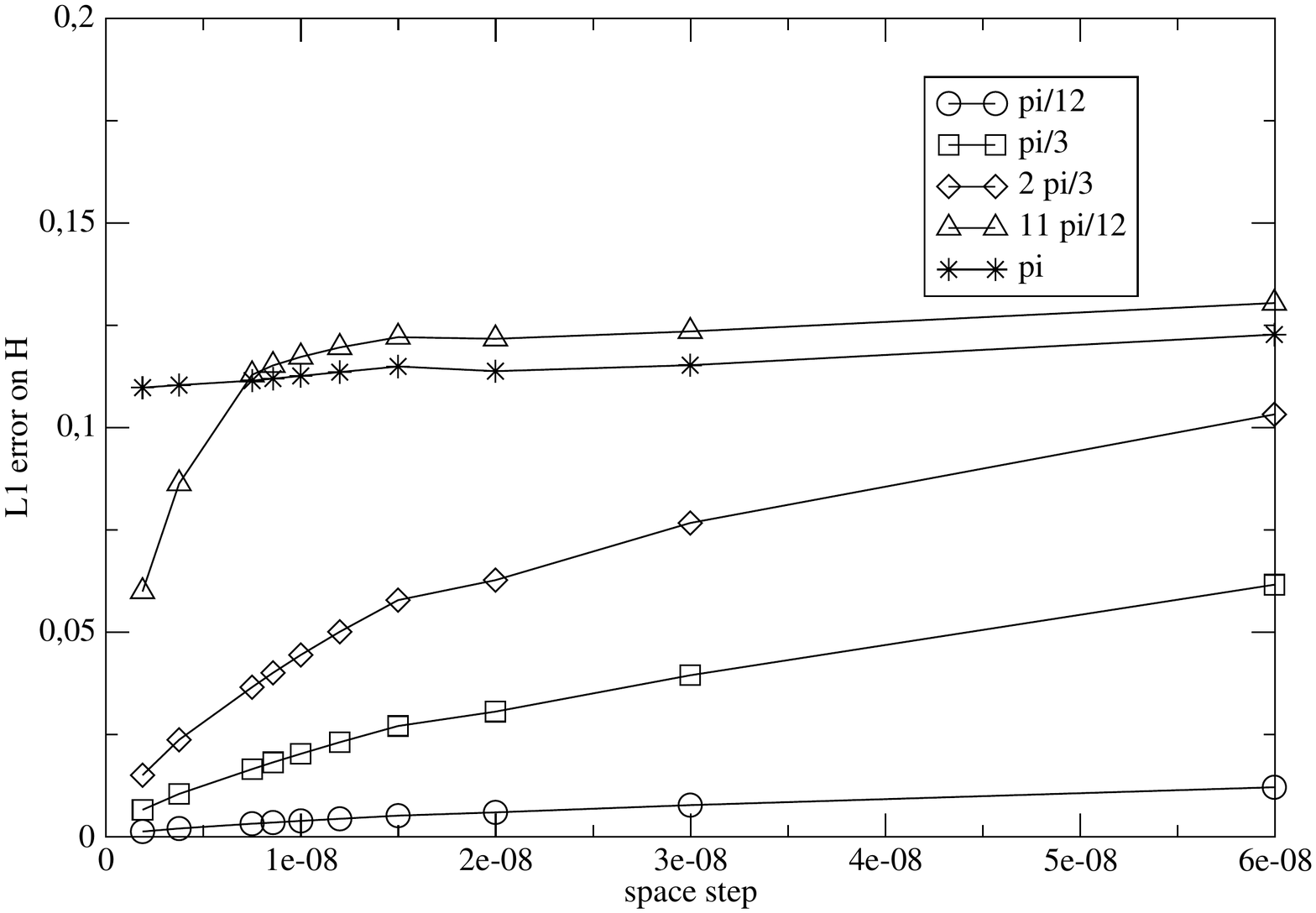}
\caption{$L^1$ relative error for $D$ (left) and $H$ (right) with respect to $\dx$: convergence holds except when reducing to a TM field.}
\label{xm44}
\end{figure}
\begin{figure}
\centering
\includegraphics[width=0.7\linewidth,angle=0]{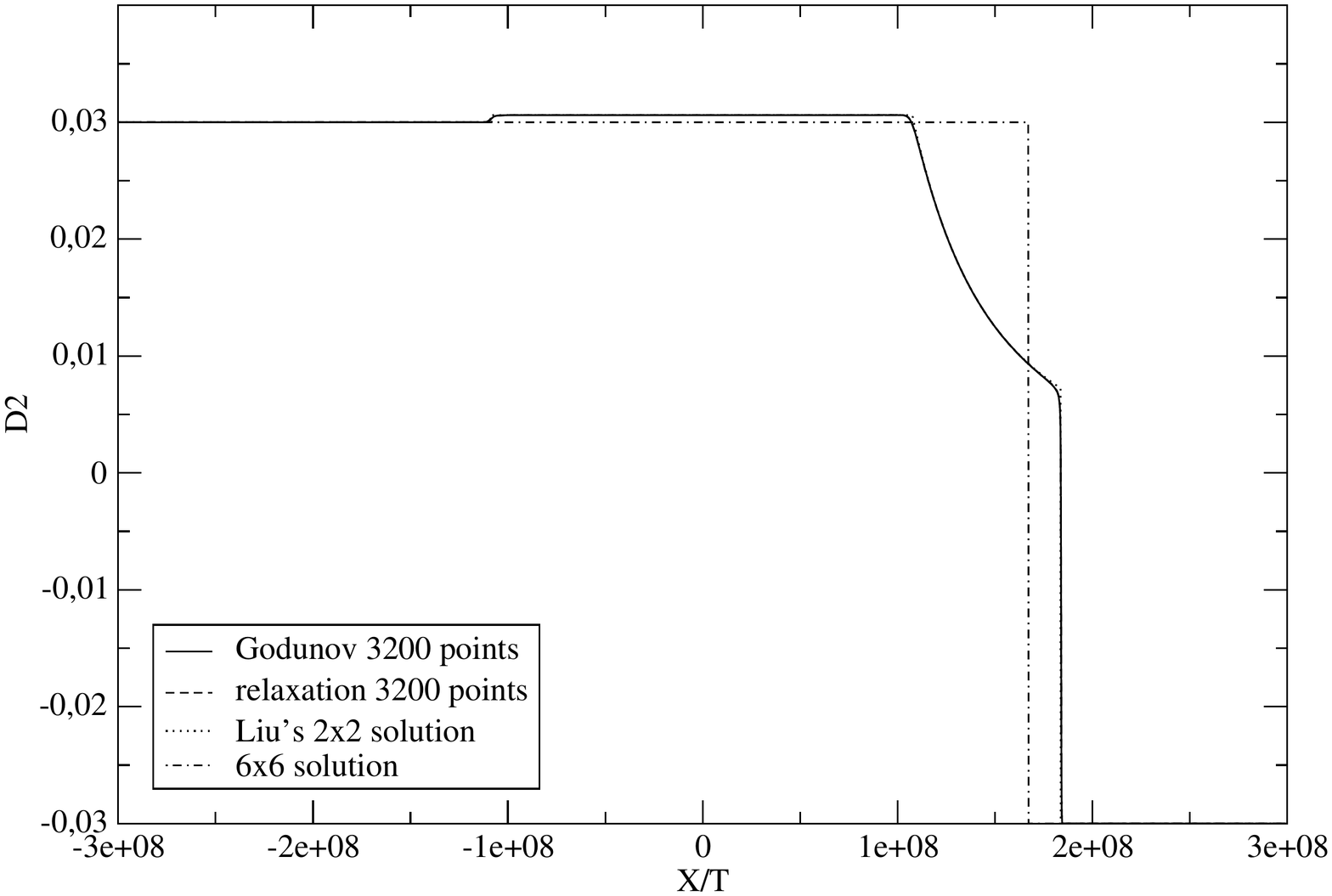}
\caption{D-component: Godunov and Kerr-Debye relaxation scheme both compute Liu's solution. H-component is similar.}
\label{xm33}
\end{figure}
\subsection{Two-dimensional cases}\label{two-d}
We restrict ourselves to computations of Transverse Magnetic fields on cartesian grids. As Riemann solver we take the $6 \times 6$ solution provided by theorem \ref{exis-sol}. We have also tested the $3 \times 3$ solution provided by theorem \ref{exis-sol2DTM}, but, as one can guess in view of one-dimensional tests, this solver gives the same results as the $6 \times 6$ one. The scheme can be written as
\[
\left\{
\begin{array}{l}
D^{n+1}_{1,ij}=D^{n}_{1,ij}+\frac{\dt}{\dy}\left(H^{n}_{3,i,j+\ud}-H^{n}_{3,i,j-\ud}\right)\\
D^{n+1}_{2,ij}=D^{n}_{2,ij}-\frac{\dt}{\dx}\left(H^{n}_{3,i+\ud,j}-H^{n}_{3,i-\ud,j}\right)\\
H^{n+1}_{3,ij}=H^{n}_{3,ij}- \frac{\dt}{\mu_0\dx}\left(E^{n}_{2,i+\ud,j}-E^{n}_{2,i-\ud,j}\right) +  \frac{\dt}{\mu_0\dy}\left(E^{n}_{1,i,j+\ud}-E^{n}_{1,i,j-\ud}\right).
\end{array}
\right.
\]
\begin{figure}
\centering
\includegraphics[width=0.3\linewidth,angle=0]{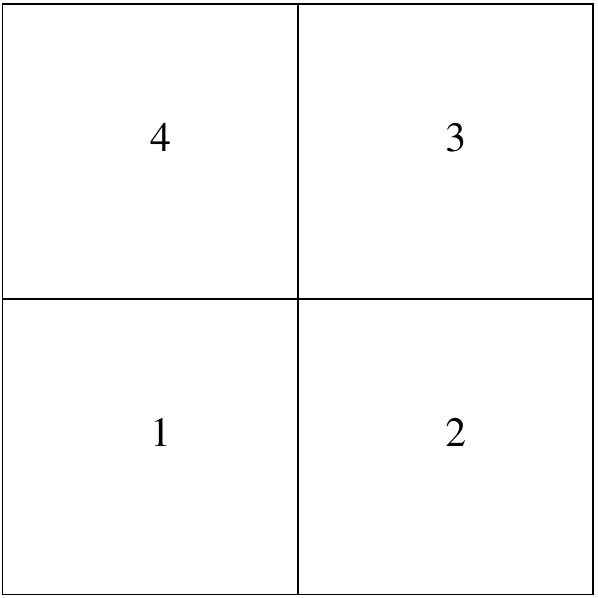}
\caption{Square partition.}
\label{square}
\end{figure}
As a first test we consider a square divided into four quadrants numbered as in Figure \ref{square}. On square $i$ we take $u^{(i)}$ as initial data with
\[
u^{(1)}=\left(
\begin{array}{c}
\delta_m \\ \delta_m \\ H_3^{(1)}
\end{array}
\right),\; u^{(2)}=\left(
\begin{array}{c}
\delta_m \\ \delta_p \\H_3^{(2)}
\end{array}
\right),\;  u^{(3)}=\left(
\begin{array}{c}
\delta_p \\ \delta_p \\ H_3^{(3)}
\end{array}
\right),\;  u^{(4)}=\left(
\begin{array}{c}
\delta_p \\ \delta_m \\ H_3^{(4)}
\end{array}
\right),
\]
in such a way that 
\begin{itemize}
\item $u^{(1)}$ and $u^{(2)}$ are connected by a Lax 5-shock, 
\item $u^{(4)}$ and $u^{(3)}$ are connected by a 2-rarefaction.
\end{itemize}
The computation is performed for a time $T=10$ femtoseconds, on a square $\Omega=]-cT,cT[^2$ with a $400 \times 400$ cartesian mesh, that is about 1320 time steps. 

Our data are divergence free but this property is not preserved by the scheme, even if for each interface the solver is divergence free. This 2D feature has already been reported in the context of MHD where it can lead to a complete blow up of the numerical solution. In our case, the results seem to be correct. The numerical ratio between $div (D)$ and $\nabla D$ is around $10^{-3}$:
\[ \int_\Omega |div D (x,y,t)|dx \, dy \leq 0.004   \int_\Omega |\nabla D (x,y,t)|dx \, dy \, . \]
This ratio remained in the same range for all the performed tests.

In figure \ref{quadr-D1D2}, the isovalues of $D_1$ and $B$ are shown. We do not represent those of $D_2$, they are in the same spirit. Near the boundaries, the problem is one-dimensional. When $y$ is fixed, we retrieve the 5-shock and the 2-rarefaction, see figure \ref{courbes1D}-left for a comparison with the exact solution near the top boundary. For fixed $x$, in view of figure \ref{quadr-D1D2}, one could think that also a single rarefaction and a single shock occur, but this is not true. The exact solution is composed at left by a 2-rarefaction and a (small) 5-shock, while at right we have a (small) 2-rarefaction and a 5-shock. Our two-dimensional computation retrieves all those waves, see figure \ref{courbes1D}-right for the right side. 
\begin{figure}
\centering
\includegraphics[width=0.49\linewidth,angle=0]{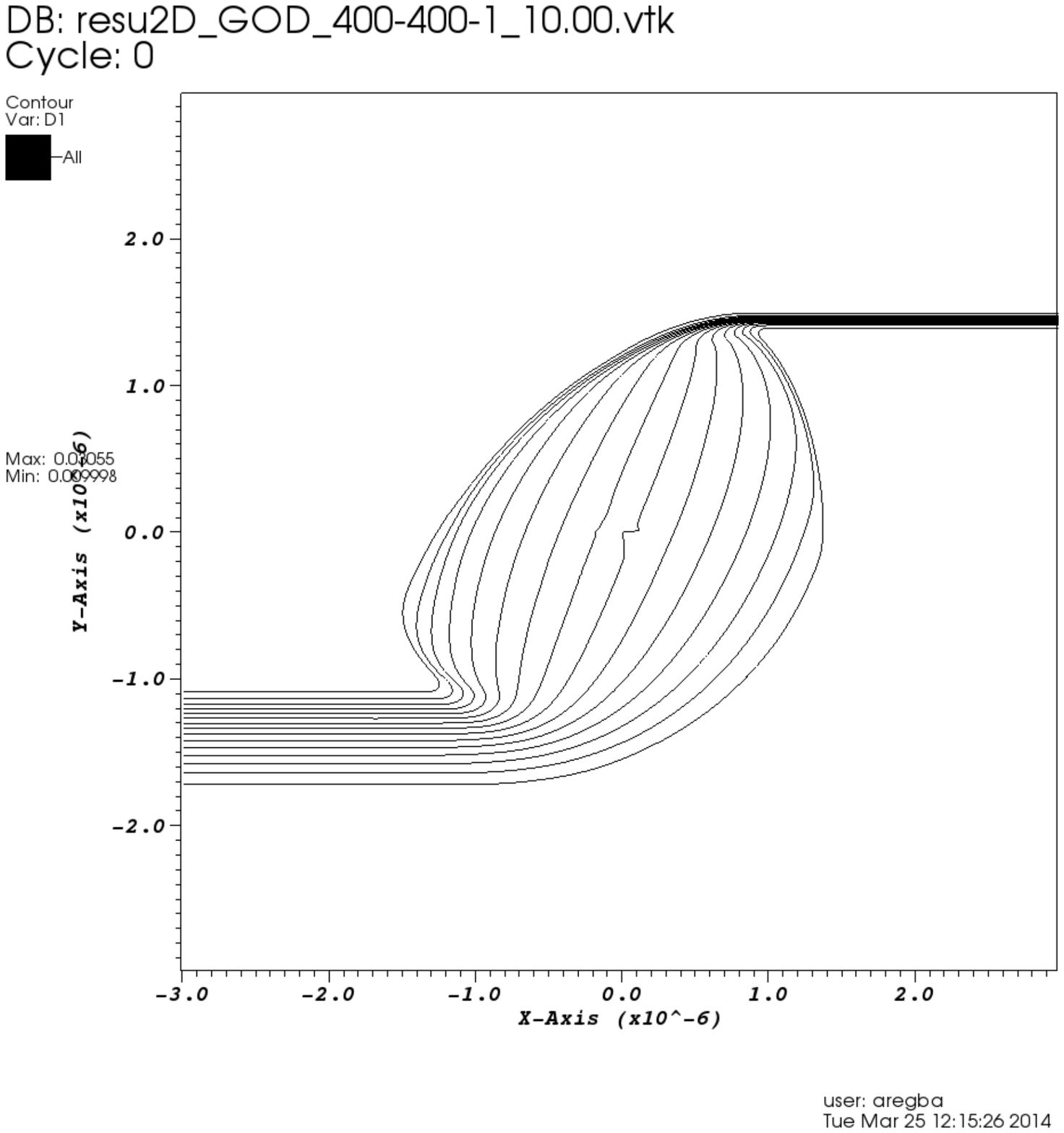}
\includegraphics[width=0.49\linewidth,angle=0]{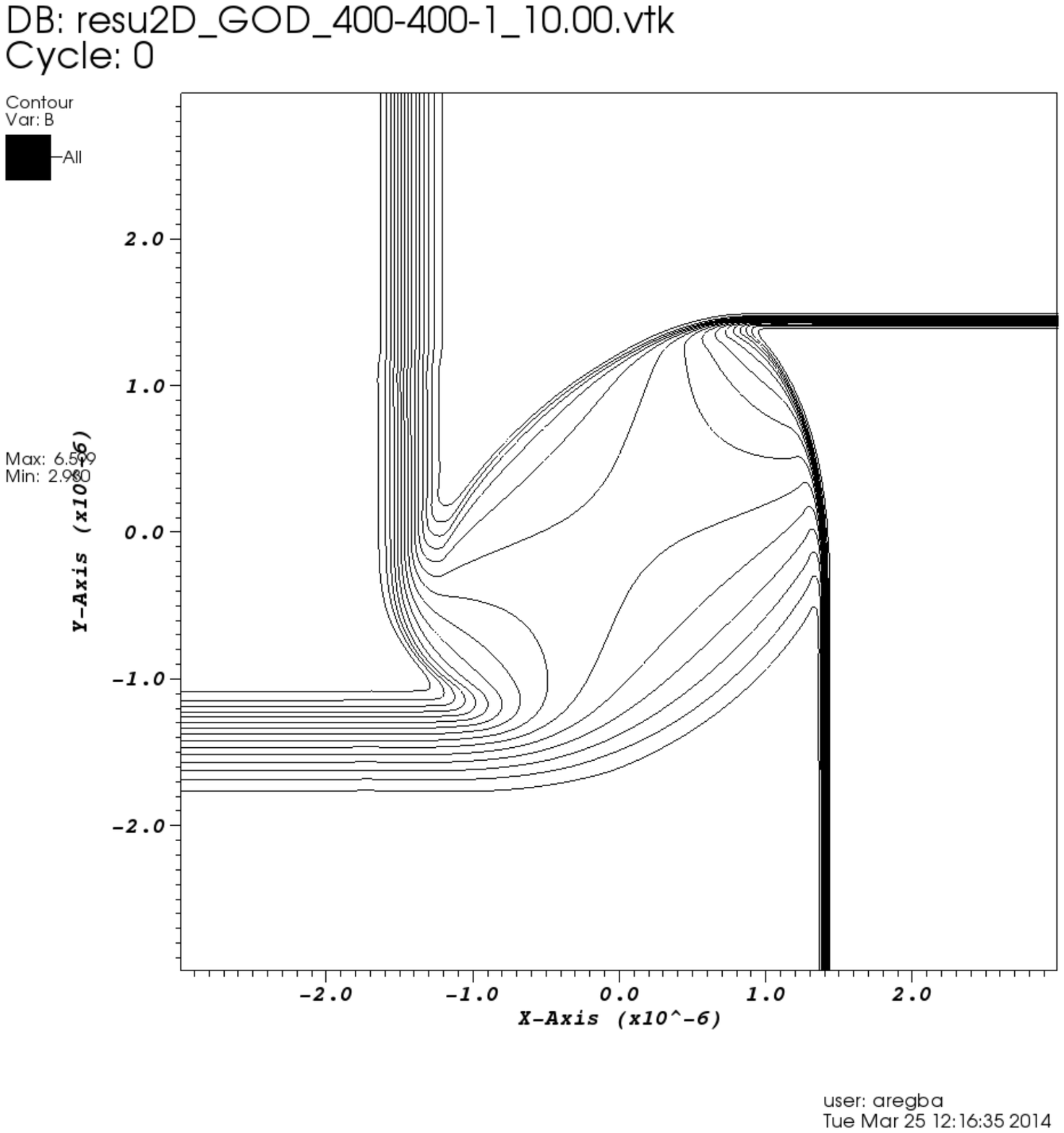}
\caption{2D Riemann problem: isovalues of $D_1$ (left) and ${B=\mu_0 H}$ (right). $D_2$ is not represented.}
\label{quadr-D1D2}
\end{figure}
\begin{figure}
\centering
\includegraphics[width=0.49\linewidth,angle=0]{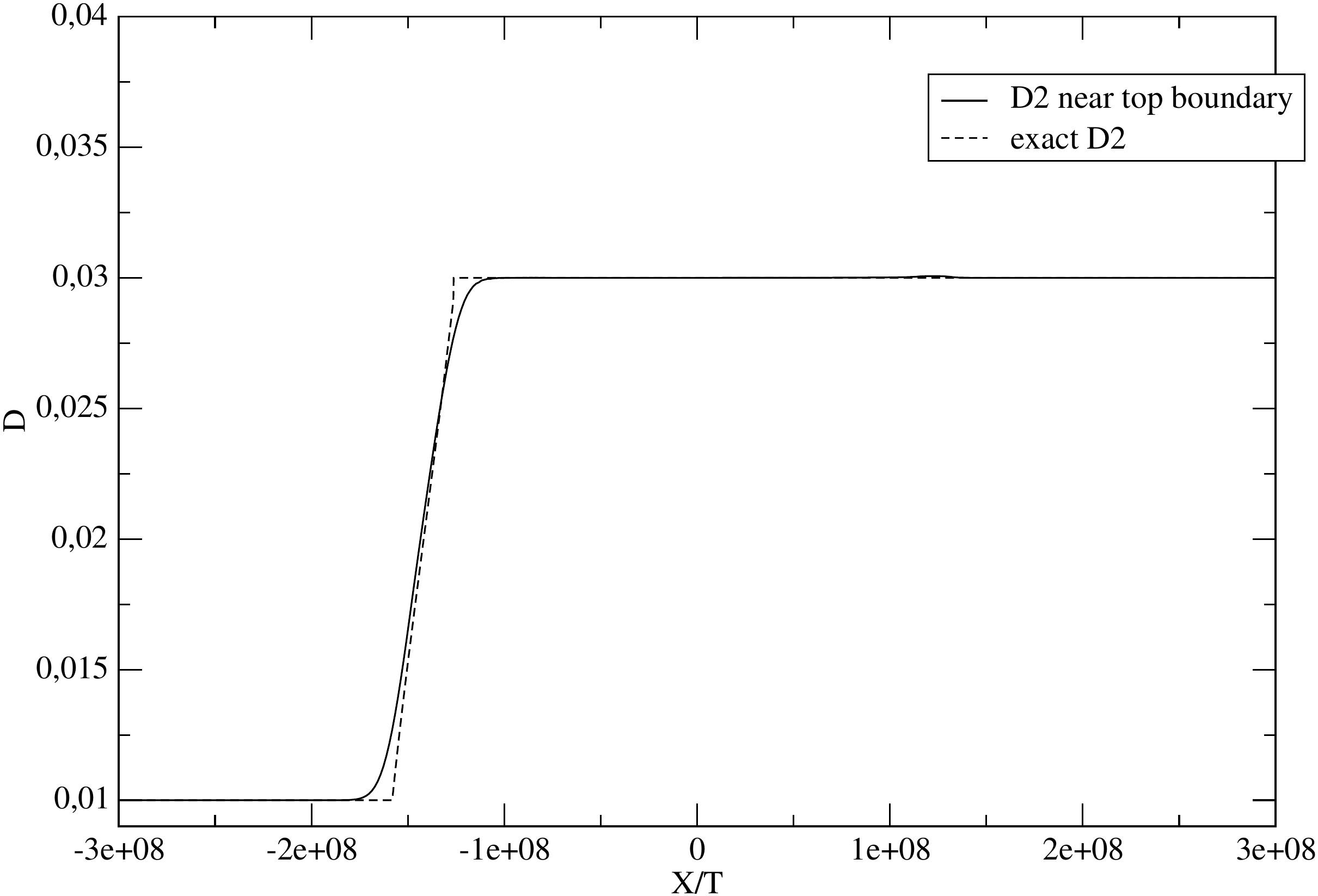}
\includegraphics[width=0.49\linewidth,angle=0]{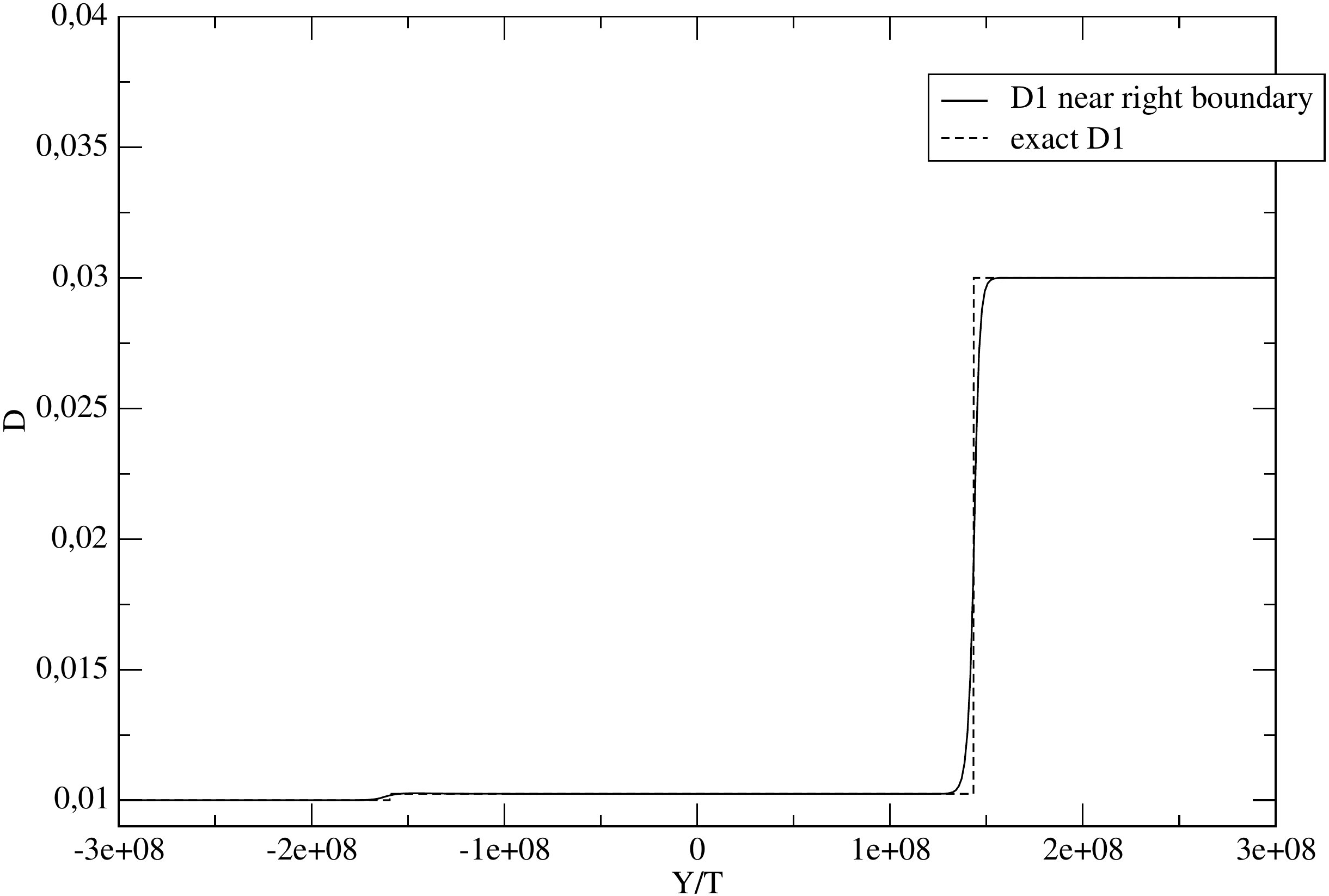}
\caption{2D Riemann problem: near the boundaries, the solution is 1D. Near the top, a single rarefaction for $(D_1,H_3)$; near the right boundary, a small rarefaction and a shock for $(D_2,H_3)$.}
\label{courbes1D}
\end{figure}
\\The second test is taken from an article by R.-W. Ziolkowski and J.B. Judkins \cite{ziol-jud}. An ultrashort pulsed optical beam is generated by a Gaussian waited magnetic field imposed at the left boundary of a rectangular domain $]0,X[ \times ]-Y,Y[$:
\[ H_3(0,y,t)=\mu_0^{-1} B_0 \left(1- \cos\left(\frac{2 \pi t}{T} \right)\right) \exp\left(- \frac{y^2}{w^2} \right) \; \; {\rm if} \; \; t \in [0,T], \quad 0 \; \; {\rm else}.
\]
The amplitude $B_0=6.087$ Tesla, the period $T=20$fs, the initial waist $w=10 \, \mu m$ are fixed. In the cited article, the response time $\tau$ of the material is not zero, and the authors solve Kerr-Debye equations (\ref{KD3D}) by a finite-difference time-domain (FDTD) method. They study self-focusing phenomena occuring in such cases. Those results have been retrieved in \cite{thhuyn} by a finite element method, and in \cite{kanso} with a finite volume scheme of which (\ref{fluxkdkd2D})-(\ref{schemkd2D}) is the relaxed Kerr limit. Also in \cite{kanso}, the Kerr limit $\tau=0$ has been investigated.  Here we compare the results obtained by Godunov scheme with those obtained in \cite{kanso}. 

The symmetry of the problem allows us to compute the field only in the domain $\Omega=]0,X[ \times ]0,Y[$. The self-focusing phenomenon can be detected by studying the time evolution of the maximal electric intensity $I(t)=\max_{\Omega}|E|^2$. After decreasing during a rather long time, by a strong interaction between the components of $E$, this quantity increases to reach a local maximum and then decreases again. As the creation of shocks dissipates energy, this maximum is less important when the response time is zero (Kerr model) than for Kerr-Debye model, but one can still observe it. In the present case, the local maximum is reached at time $t=109.2$ femtoseconds. In Figure \ref{xmcourbe2D}-left, we zoom on the time evolution of $I(t)$. We remark that the relaxation scheme (\ref{fluxkdkd2D})-(\ref{schemkd2D}) and Godunov scheme give the same result. We just represent the isolines of $|E|^2$ for Godunov scheme (Figure \ref{xmcourbe2D}-right), they are nearly the same as those obtained by the relaxation scheme. As can be seen on this Figure, we have both self-focusing and shock creation, which means that physical situations require an efficient computation of solutions with shocks.
\begin{figure}
\centering
\includegraphics[width=0.49\linewidth,angle=0]{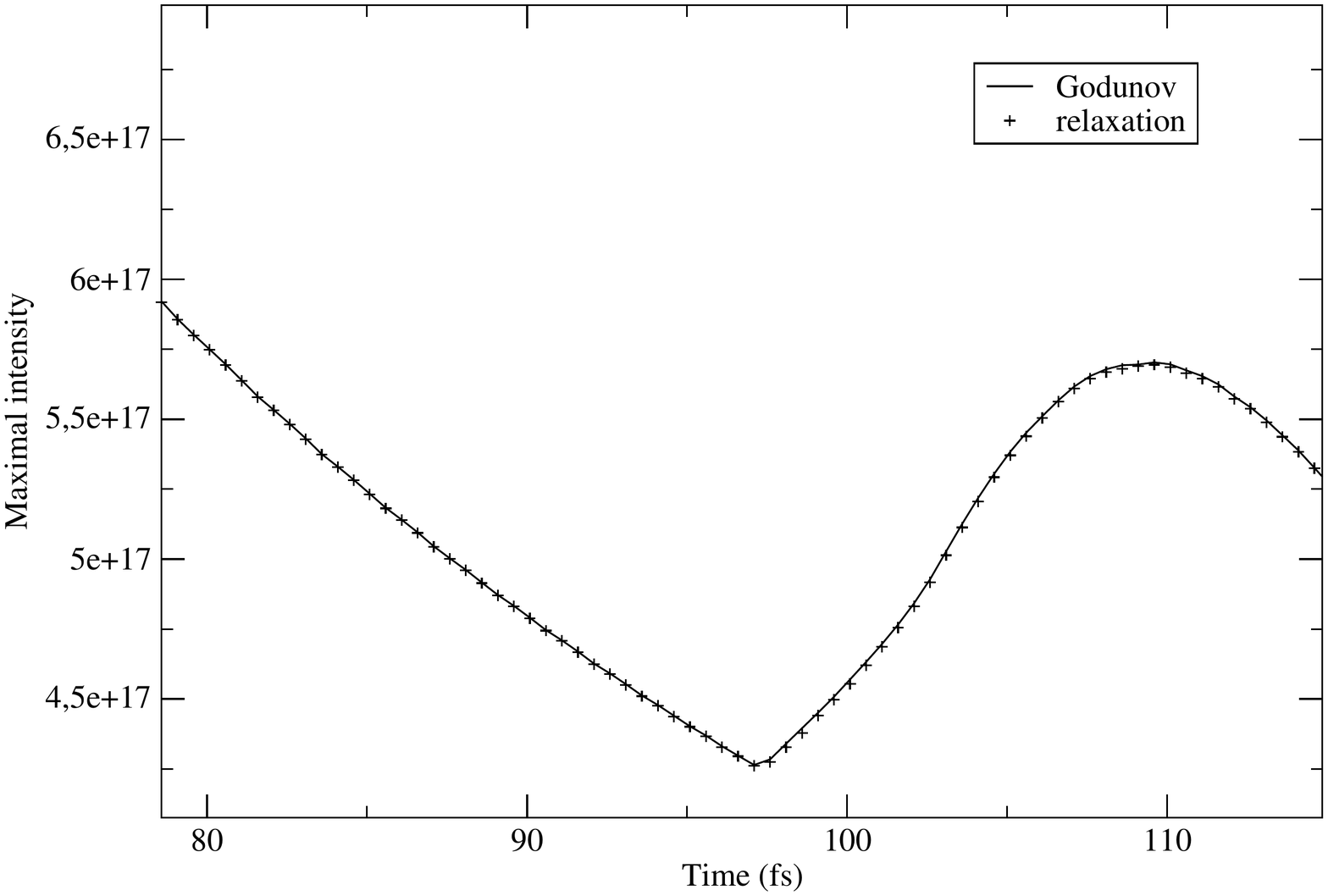}
\includegraphics[width=0.49\linewidth,angle=0]{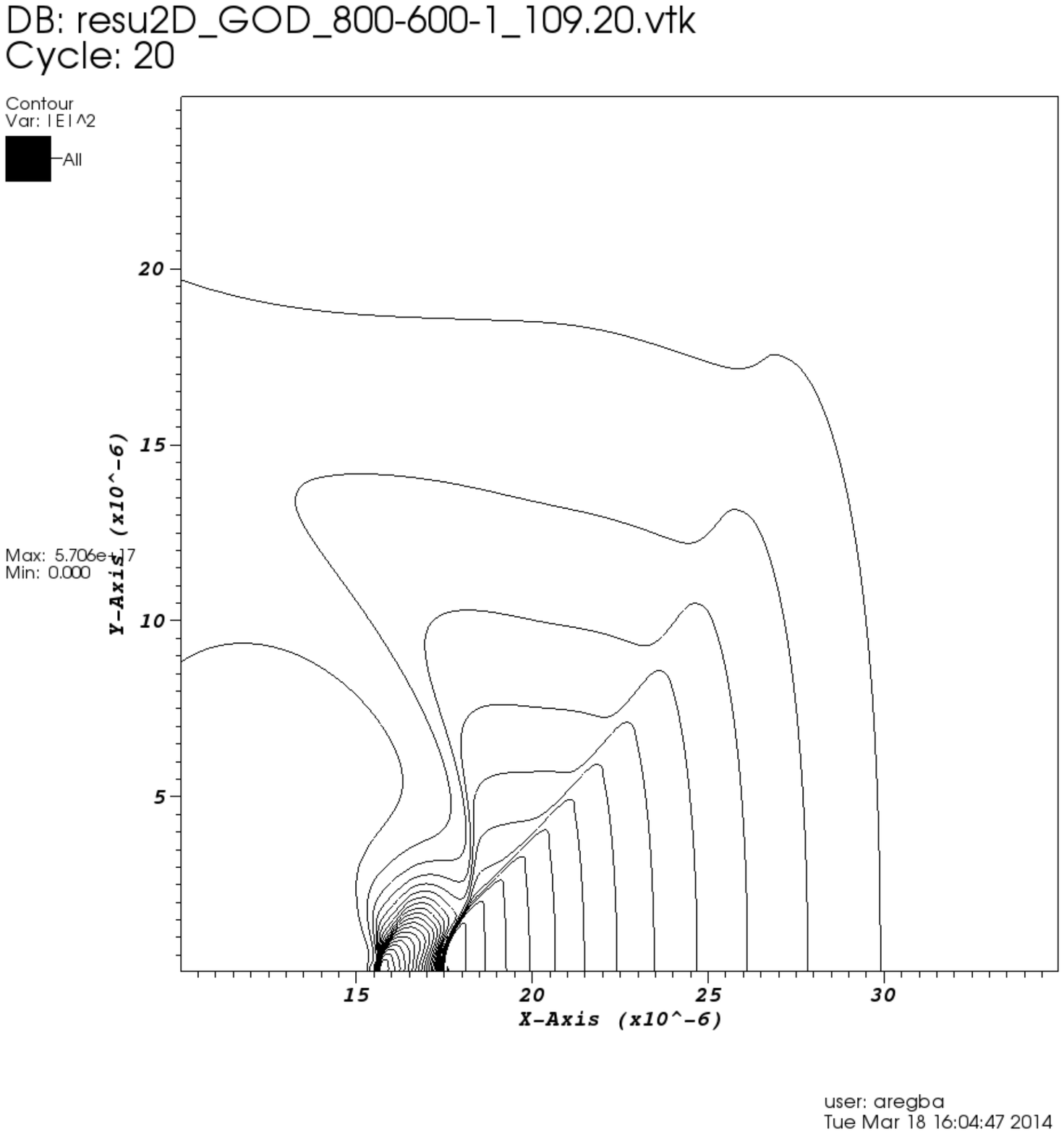}
\caption{2D case. Left: maximal intensity $|E|^2$ in the computational domain with respect to time. Right: maximal intensity at self-focusing time.}
\label{xmcourbe2D}
\end{figure}
\section{Conclusion}
We have been able to solve the Riemann problem for the $6 \times 6$ Kerr system. The multiplicity of the eigenvalues is not constant and the characteristic fields 2 and 5 are neither genuinely nonlinear, nor linearly degenerate. Nevertheless, in all cases, we can construct a unique Lax solution. For Transverse Magnetic data, this solution is Transverse Magnetic and does not coincide with Liu's solution of the reduced $3 \times 3$ system. This allows us to point out the non uniqueness of selfsimilar weak entropy solutions of Kerr system. 

From the numerical viewpoint, the $6 \times 6$ Lax solution has been implemented as an exact Riemann solver for Godunov scheme. Numerical experiments have been performed in one and two dimensions, including realistic physical cases. The results are very close to those obtained by the Kerr-Debye relaxation scheme coming from the non-zero response time model. In the particular case of coexistence of two entropy solutions, always the more dissipative Liu's solution is reached by our schemes. This may be due to numerical viscosity. In the physical case of an ultrashort pulsed optical beam, our results are consistent with those of the literature, and the well known self-focusing phenomenon is retrieved.

The results of Godunov and Kerr-Debye relaxation schemes are very close. The relaxation scheme is completely explicit. Due to the resolution of a nonlinear algebraic equation for each cell, Godunov scheme is more expensive in terms of CPU time but by construction, it allows us to compute weak solutions of Kerr system, even when they contain shocks. Along with the physically consistent relaxation scheme, we now have two reliable computational methods for Kerr system.
\section{Annex: Kerr-Debye relaxation scheme}\label{KDebye}
System (\ref{KD3D}) is hyperbolic with eigenvalues
\[ \mu_1=\mu_2=-\mu, \quad \mu_3=\mu_4=\mu_5=0, \quad \mu_6=\mu_7=\mu\]
with $\mu=\displaystyle \frac{c}{\sqrt{1+\chi}}$. Moreover, all the characteristic fields are linearly degenerate. These properties are useful to design a numerical approximation of (\ref{Kerr3D}), following the classical projection-transport technique. At every time step, one first projects the solution onto equilibrium by setting $\chi= \epsilon_r p^2(|D|)$, then the homogeneous system related to (\ref{KD3D}) is solved. As we use the finite volume method, we just have to know the solution of the Riemann problem to find the numerical fluxes at each interface, that is the approximation of $- \omega \pv H$ and $\mu_0^{-1}\omega \pv E$. The Riemann problem is easy to solve because we have only contact discontinuities here. Denoting $U_-$, $U_+$ the left and right initial states, $r_\pm=\sqrt{1+\ki_\pm}$, 

$E_\pm=\displaystyle \frac{D_\pm}{\epsilon_0(1+\chi_\pm)}$, we find (see \cite{kanso} for a proof in the tranverse magnetic case):
\[ 
\varphi(U_-,U_+,\om)=\left( \begin{array}{c}
-\omega \pv \displaystyle\frac{r_+H_- + r_-H_+}{r_++r_-}+\omega \pv \left( \omega \pv \frac{r_+r_-(E_+-E_-)}{c \mu_0(r_++r_-)} \right)\\ \\
\omega \pv \displaystyle\frac{r_+ E_+ + r_- E_-}{\mu_0(r_++r_-)} + c \, \om \pv \left(\om \pv \displaystyle\frac{H_+-H_-}{r_++r_-}\right)
\end{array}
\right)
\]
In one space dimension, we set $\overline{U_i}=(u^n_i,\epsilon_r p^2(|D^n_i|)$, $\om=(1,0,0)=e_1$, 
\begin{equation}\label{fluxkd1D} 
F^n_{i+\ud}=\varphi(\overline{U_i},\overline{U_{i+1}},e_1)
\end{equation}
and 
\begin{equation}\label{schemkd1D} 
u^{n+1}_i=u^n_i-\frac{\dt}{\dx}(F^n_{i+\ud}-F^n_{i-\ud}).
\end{equation}
Notice that this implies that $D^{n+1}_{i,1}=D^n_{i,1}$ and $H^{n+1}_{i,1}=H^n_{i,1}$ for all $i$ and $n$, which means that $\di D$ and $\di H$ are constant along the computation. 

In two space dimensions, for a cartesian mesh, we set
\begin{equation}\label{fluxkdkd2D} 
F^n_{i+\ud,j}=\varphi(\overline{U_{i,j}},\overline{U_{i+1,j}},e_1), \quad G^n_{i,j+\ud}=\varphi(\overline{U_{i,j}},\overline{U_{i,j+1}},e_2)
\end{equation}
and the scheme reads as 
\begin{equation}\label{schemkd2D} 
u^{n+1}_{i,j}=u^n_{i,j}-\frac{\dt}{\dx}(F^n_{i+\ud,j}-F^n_{i-\ud,j})-\frac{\dt}{\dy}(G^n_{i,j+\ud}-G^n_{i,j-\ud}).
\end{equation}

\end{document}